\documentclass[11pt]{amsart}
\usepackage[british]{babel}

\usepackage{amssymb,latexsym}
\usepackage{comment}
\usepackage{titletoc}

\usepackage{xr-hyper}
\usepackage{amsfonts, amsmath, amsthm, amssymb, hyperref, fancybox, graphicx, color, times, babel}
\usepackage[utf8]{inputenc}
\usepackage{multirow}
\usepackage{color}

\usepackage[normalem]{ulem}

\usepackage[all,cmtip]{xy}
\newtheorem {thm}{Theorem}
\newtheorem* {thm*}{Theorem}

\newtheorem {cor}[thm]{Corollary}

\newtheorem* {cor*}{Corollary}
\newtheorem {lem}[thm]{Lemma}

\newtheorem {prop}[thm]{Proposition}
\newtheorem {defi}[thm]{Definition}
\newtheorem {rem}[thm]{Remark}

\theoremstyle{definition}
\newtheorem {exa}[thm]{Example}
\newtheorem* {conj*}{Conjecture}

\newtheorem* {quest*}{Question}

\newcommand{\lval}{v_\ell}

\newcommand{\Q}{\mathbb{Q}}
\newcommand{\Z}{\mathbb{Z}}

\newcommand{\GL}{\mathrm{GL}}

\numberwithin{equation}{section}

\begin{document}

\title{The $1$-Eigenspace for matrices in $\GL_2(\Z_\ell)$}
\author{Davide~Lombardo and Antonella~Perucca}

\begin{abstract}
Fix some prime number $\ell$ and consider an open subgroup $G$ either of $\GL_2(\Z_{\ell})$ or of the normalizer of a Cartan subgroup of $\GL_2(\Z_{\ell})$. The elements of $G$ act on $(\Z/\ell^n \Z)^2$ for every $n\geqslant 1$ and also on the direct limit, and we call $1$-Eigenspace the group of fixed points. We partition $G$ by considering the possible group structures for the $1$-Eigenspace and show how to evaluate with a finite procedure the Haar measure of all sets in the partition. The results apply to all elliptic curves defined over a number field, where we consider the image of the $\ell$-adic representation and the Galois action on the torsion points of order a power of $\ell$.
\end{abstract}

\address[]{Dipartimento di Matematica, Università di Pisa, Largo Pontecorvo 5, 56127 Pisa, Italy}
\email[]{davide.lombardo@unipi.it}
\address[]{Universit\"at Regensburg, Universit\"atsstra{\ss}e 31,
93053 Regensburg, Germany}
\email[]{mail@antonellaperucca.net}
\maketitle

\section{Introduction}

Fix a prime number $\ell$, and let $G$ be an open subgroup of either $\GL_2(\mathbb{Z}_\ell)$ or the normalizer of a (possibly ramified) Cartan subgroup of $\GL_2(\mathbb{Z}_\ell)$. This general framework can be applied to elliptic curves defined over a number field, where $G$ is the image of the $\ell$-adic representation.
We identify an element of $G$ with an automorphism of the direct limit in $n$ of $(\Z/\ell^n \Z)^2$: for elliptic curves this means considering the Galois action on the group of torsion points whose order is a power of $\ell$.

We equip $G$ with its Haar measure, normalized so as to assign volume one to $G$, and we compute the measure of subsets of $G$ of arithmetic interest.
For $M \in G$, we call \textit{$1$-Eigenspace of $M$} the subgroup of fixed points of $M$ for its action on the direct limit $\varinjlim_n (\Z/\ell^n\Z)^2$. This leads to partitioning $G$ into subsets according to the group structure of the 1-Eigenspace.
More specifically, the matrices whose $1$-Eigenspace is an infinite group form a subset of $G$ that has Haar measure zero, so we only investigate the possible finite group structures. For all integers $a,b\geqslant 0$ we consider the set 
\begin{equation*}
\mathcal M_{a,b}:=\{M\in G: \, \ker (M-I) \simeq  \Z/{\ell^a}\Z\times \Z/{\ell^{a+b}}\Z  \}
\end{equation*}
and its Haar measure in $G$, which is well-defined for each pair $(a, b)$ and that we call $\mu_{a,b}$. The aim of this paper is to show that the whole countable family $\mu_{a,b}$ can be effectively computed:

\begin{thm}\label{main-thm}
Fix a prime number $\ell$ and an open subgroup $G$ of either $\GL_2(\mathbb{Z}_\ell)$ or the normalizer of a (possibly ramified) Cartan subgroup of $\GL_2(\mathbb{Z}_\ell)$. It is possible to compute the whole family $\{\mu_{a,b}\}$ for $(a,b)\in \mathbb N^2$ with a finite procedure. More precisely, we can partition $\mathbb N^2$ in finitely many subsets $S$ (as in Definition \ref{admissible} and explicitly computable) such that the following holds: there is some (explicitly computable) rational number $c_S\geqslant 0$ such that for every $(a,b)\in S$ we have 
$$\mu_{a,b}= c_S\cdot \ell^{-(\dim(G) a+b)}$$
where the dimension of $G$ is either $4$ or $2$, according to whether $G$ is open in $\GL_2(\mathbb{Z}_\ell)$ or in the 
normalizer of a Cartan subgroup. The sets $S$ and the constants $c_S$ may depend on $\ell$ and $G$.
\end{thm}

Some explicit results are as follows:

\begin{thm}\label{thm-countGL2}
For $\GL_2(\mathbb Z_{\ell})$, we have:
$$\mu_{a,b}= \left\{ 
\begin{array}{llll}
\dfrac{\ell^3-2\ell^2-\ell+3}{(\ell-1)^2 \cdot (\ell+1)} & \text{if $a=0, b=0$} \medskip \\
\dfrac{\ell^2-\ell-1}{\ell(\ell-1)}\cdot \ell^{-b} & \text{if $a=0, b>0$} \medskip \\
\ell^{-4a} & \text{if $a>0, b=0$} \medskip \\
(\ell + 1)\cdot \ell^{-4a-b-1} & \text{if $a>0, b>0$ .} \\
\end{array}\right.$$
\end{thm}

\begin{thm}\label{thm-countsplitnonsplitCartan}
For a Cartan subgroup of $\GL_2(\mathbb Z_{\ell})$ which is either split or nonsplit (see Definition \ref{Cartan}) we respectively have:
$$\mu_{a,b} = \left\{ 
\begin{array}{lllll}
\dfrac{(\ell-2)^2}{(\ell-1)^2}   & \text{if $a=0, b=0$} \medskip  \\
\dfrac{2(\ell-2)}{\ell-1} \cdot \ell^{-b} & \text{if $a=0, b>0$} \medskip \\
\ell^{-2a} & \text{if $a>0, b=0$} \\
2\cdot \ell^{-2a-b} & \text{if $a>0, b>0$} \\
\end{array}\right.
\qquad
\mu_{a,b} = \left\{ 
\begin{array}{lllll}
\dfrac{\ell^2-2}{\ell^2-1}  & \text{if $a=0, b=0$}\medskip \\
\ell^{-2a}  & \text{if $a>0, b=0$}\medskip \\
0 & \text{if $b>0$} \,. \\
\end{array}\right.
$$
\end{thm}

\begin{thm}\label{thm-countnormalizerCartan}
For the normalizer of a split or nonsplit Cartan subgroup $C$ of $\GL_2(\Z_\ell)$ we have 
$$\mu_{a,b} =\frac{1}{2} \cdot \mu_{a,b}^{C} + \frac{1}{2} \cdot \mu^*_{a,b}$$
where $\mu_{a,b}^C$ is the Haar measure in $C$ of $\mathcal M_{a,b}\cap C$ (which can be read off Theorem \ref{thm-countsplitnonsplitCartan}) and where  we set 
$$\mu^*_{a,b}=\left\{ 
\begin{array}{lllll}
\dfrac{\ell-2}{\ell-1} & \text{if $a=0, b=0$} \medskip  \\
\ell^{-b} & \text{if $a=0, b>0$} \medskip \\
0 & \text{if $a>0$ .}\\
\end{array}\right.$$
\end{thm}

The Haar measure $\mu_{a,b}$ can computed as the limit in $n$ of the ratio $\# \mathcal{M}_{a,b}(n) / \# G(n)$, where for a subset $X$ of $\GL_2(\Z_\ell)$ the symbol $X(n)$ denotes the image of $X$ in $\GL_2(\Z/\ell^n\Z)$. 
For fixed $a$ and $b$, the quantity $\# \mathcal{M}_{a,b}(n) / \# G(n)$ stabilizes for $n$ sufficiently large by the higher-dimensional version of Hensel's Lemma. However, since we cannot fix a single value of $n$ which is good for every pair $(a,b)$, we need technical results about counting the number of lifts of any given matrix in $\GL_2(\Z/\ell^n\Z)$ to $\GL_2(\Z/\ell^{n+1}\Z)$.

The structure of the paper is as follows. In Section \ref{sect:PreliminariesCartan} we define Cartan subgroups of $\GL_2(\mathbb{Z}_\ell)$ in full generality and prove a classification result which might be of independent interest. In Section \ref{sec-1ES} we prove general results about the group structure of the $1$-Eigenspace and set the notation for the subsequent sections. These contain further results, in particular Theorem \ref{thm-lifts} (about the reductions of $\mathcal M_{a,b}$) and the two technical results  Theorems \ref{thm:GeneralLift} and \ref{thm:LiftsNormalizer}.
Finally, the last section is devoted to the proof of Theorems 1 to 4. In \cite{LombardoPeruccaRedEC} we apply the results of this paper to solve a problem about elliptic curves:

\begin{rem}
Let $E$ be an elliptic curve defined over a number field $K$. If $\ell$ is a prime number and $E[\ell^\infty]$ is the group of $\overline{K}$-points on $E$ of order a power of $\ell$, we have general results and a computational strategy for:
\begin{itemize}
\item classifying the elements in the image of the $\ell$-adic representation according to the group structure of the fixed points in $E[\ell^\infty]$;
\item computing the density of reductions such that the $\ell$-part of the group of local points has some prescribed group structure, for the whole family of possible group structures.
\end{itemize}
 \end{rem}

\section{Cartan subgroups of $\GL_2(\mathbb{Z}_\ell)$}\label{sect:PreliminariesCartan}

\subsection{General definition of Cartan subgroups}
Classical references are \cite[Chapter 4]{MR1102012} and \cite[Section 2]{MR0387283}. 
Let $\ell$ be a prime number and $F$ be a reduced $\Q_\ell$-algebra of degree 2 with ring of integers $\mathcal{O}_F$.
Concretely, $F$ is either a quadratic extension of $\Q_\ell$, or the ring $\Q_\ell^2$ (in the latter case we define the $\ell$-adic valuation as the minimum of those of the two coordinates and by $\mathcal{O}_F$ we mean the valuation ring $\Z_\ell^2$).
Let furthermore $R$ be a $\Z_\ell$-order in $F$, by which we mean a subring of $F$ (containing $1$) which is a finitely generated $\Z_\ell$-module and satisfies $\mathbb{Q}_\ell R=F$ (i.e. $R$ spans $F$ over $\mathbb{Q}_\ell$).

The \emph{Cartan subgroup $C$ of $\GL_2(\Z_\ell)$ associated with $R$} is the group of units of $R$: the embedding $R^\times \hookrightarrow \GL_2(\Z_\ell)$ is given by fixing a $\Z_\ell$-basis of $R$ and considering the left multiplication action of $R^\times$.  The Cartan subgroup is only well-defined up to  conjugation in $\GL_2(\Z_\ell)$ because of the choice of the basis. Writing $C_R:=\operatorname{Res}_{R/\Z_\ell}(\mathbb{G}_m)$, where $\operatorname{Res}$ is the Weil restriction of scalars, we have $C =C_R (\mathbb Z_{\ell})$, provided that the Weil restriction is computed using the same $\Z_\ell$-basis for $R$. 

Equivalently, a Cartan subgroup of $\GL_2(\Z_\ell)$ can be described as follows: there exists a maximal torus $\mathcal{T}$ of $\GL_{2,\Z_\ell}$, flat over $\Z_\ell$, such that $C=\mathcal{T}(\Z_\ell)$.

\begin{defi}\label{Cartan}
We shall say that the Cartan subgroup of $\GL_2(\Z_\ell)$ associated with $R$ is:
\begin{itemize}
\item \emph{maximal}, if $\ell$ does not divide the index of $R$ in $\mathcal{O}_F$;
\item \emph{split}, if it is maximal and furthermore $\ell$ is split in $F$;
\item \emph{nonsplit}, if it is maximal and furthermore $\ell$ is inert in $F$;
\item \emph{ramified}, if it is neither split nor nonsplit.
\end{itemize}
\end{defi}

Notice in particular that \emph{unramified} means the same as \emph{either split or nonsplit}. Thus a Cartan subgroup is either split, nonsplit or ramified: a Cartan subgroup can be ramified because it is not maximal ($\ell$ divides $[\mathcal{O}_F:R]$), or because $\ell$ ramifies in $F$.
Note that we always understand `maximal'  in the sense of the above definition (in particular, even if a Cartan subgroup is not maximal, it is still the group of $\mathbb Z_{\ell}$-points of a maximal subtorus of $\GL_2$). A proper subgroup of a Cartan subgroup of $\GL_2(\Z_\ell)$ is not a Cartan subgroup in our terminology.

\begin{rem}
A strict inclusion of quadratic rings $R \hookrightarrow S$ over $\Z_\ell$ does not induce an inclusion of Cartan subgroups according to our definition. This is because the multiplication action of $R^\times$ on $R$ (resp.~of $S^\times$ on~$S$) is represented with respect to a $\Z_\ell$-basis of $R$ (resp.~$S$), and the base-change matrix relating a basis of $R$ with a basis of $S$ is not $\ell$-integral. More concretely, write $S=\Z_\ell[\omega]$ and  $R=\Z_\ell[\ell^k\omega]$ for some $k> 0$. Suppose for simplicity that $\ell \neq 2$ and $\omega^2=d \in \mathbb{Z}_\ell$, and consider  the bases $\{1,\omega\}$ and $\{1,\ell^k\omega\}$ of $S, R$ respectively. An element $a+b \ell^k \omega$ (where $a,b\in \Z_\ell$) corresponds to 
$$
\begin{pmatrix}
a & b \ell^{k} d \\ b \ell^k & a
\end{pmatrix}\in C_S(\Z_\ell) \qquad \text{and} \qquad \begin{pmatrix}
a & b \ell^{2k} d \\ b & a
\end{pmatrix} \in C_R(\Z_\ell)\,.$$
One can check that for $b \neq 0$ there is no $\Z_\ell$-integral change of basis relating these two matrices, and a similar conclusion holds for any choice of $\Z_\ell$-bases of $R, S$.
\end{rem}

For a maximal Cartan subgroup we have $R = \mathcal{O}_F$ and for a split Cartan subgroup we have $R \cong \Z_\ell^2$ and hence $C \cong (\Z_\ell^\times)^2$.

\subsection{A classification for quadratic rings}
It is apparent from the previous discussion that classifying the Cartan subgroups of $\GL_2(\Z_\ell)$ up to conjugacy is equivalent to classifying the \emph{quadratic rings} over $\Z_\ell$ (i.e.\@ the orders in integral quadratic $\mathbb{Q}_\ell$-algebras) up to a $\Z_\ell$-linear ring isomorphism. 
A Cartan subgroup is maximal if and only if the corresponding quadratic ring $R$ is the maximal order; a maximal Cartan subgroup is unramified if and only if the corresponding $\mathbb{Q}_\ell$-algebra is \'etale i.e.\@ it is either $\mathbb{Q}_\ell^2$ (in the split case) or the unique unramified quadratic extension of $\mathbb{Q}_\ell$ (in the nonsplit case). We have an \'etale $\mathbb{Q}_\ell$-algebra if and only if the $\ell$-adic valuation on $R$, normalized so that $v_\ell(\ell)=1$, takes integer values.

\begin{thm}{(Classification of quadratic rings)}\label{quadratic-rings}
If $R$ is a quadratic ring over $\Z_\ell$ then there exist a $\Z_\ell$-basis $(1,\omega)$ of $R$ and parameters $(c,d)$ in $\Z_\ell$ satisfying $\omega^2=c\omega+d$ and such that one of the following holds: $c=0$ (and hence $d\neq 0$); $\ell=2$, $c=1$, and $d$ is either zero or odd.
\end{thm}
\begin{proof}
Let $(1,\omega_0)$ be a $\Z_\ell$-basis of $R$ and write $\omega_0^2=c_0\omega_0+d_0$ for some $c_0,d_0 \in \Z_\ell$. If $\ell$ is odd or $c_0$ is even, we set $\omega=\omega_0-c_0/2$ and have parameters $(0,d_0+{c_0^2}/{4})$. If $\ell=2$ and $c_0$ is odd, we set $\omega=\omega_0-{(c_0-1)}/{2}$ and $d_1=d_0+{(c_0^2-1)}/{4}$. If $d_1$ is odd, we are done because we have $\omega^2=\omega+d_1$. If $d_1$ is even, the quadratic equation
$\omega^2=\omega+d_1$ has solutions in $\mathbb{Q}_2$ because its discriminant $1-4d_1 \equiv 1 \pmod 8$ is a square. Thus $R$ is an order in $\mathbb{Q}^2_2$ and hence it is of the form $\mathbb{Z}_2 (1,1) \oplus \mathbb{Z}_2 (0,\beta)$ for some $\beta \in \Z_2$. If $\beta$ is odd, we have $R=\mathbb{Z}^2_2$ so we  set $\omega=(0,1)$ and have parameters $(1,0)$. If $\beta$ is even, we set $\omega=(-{\beta}/{2},{\beta}/{2})$ and have parameters $(0,{\beta^2}/{4})$. \end{proof}

\subsection{Parameters for a Cartan subgroup}\label{subsec-classification}

We call the parameters $(c,d)$ as in Theorem \ref{quadratic-rings}
\emph{parameters for the Cartan subgroup} of $\GL_2(\Z_\ell)$ corresponding to $R$: they are in general not uniquely determined.

\begin{rem}\label{rmk:IntegralParameters}
Since $\Z$ is dense in $\Z_\ell$ we may assume that the parameters $(c,d)$ are integers. Indeed, one can prove that the isomorphism class of the ring $\mathbb{Z}_\ell[x]/(x^2-cx-d)$ is a locally constant function of $(c,d) \in \Z_\ell^2$ (this property is closely related to Krasner's Lemma \cite[Tag 0BU9]{stacks-project}). We also give a direct argument.
Consider first a Cartan subgroup $C$ with parameters $(0,d)$.
If $u$ is an $\ell$-adic unit, $(0,u^2d)$ are also parameters for $C$. Thus $C$ depends on $d$ only through its class in $(\mathbb{Z}_\ell \setminus \{0\})/\mathbb{Z}_\ell^{\times 2}$ (quotient as multiplicative monoids): this is isomorphic to $\mathbb{N} \times \mathbb{Z}_\ell^\times/\Z_\ell^{\times 2}$, where the first factor is the valuation and the second factor is finite (indeed $\Z_\ell^\times/\Z_\ell^{\times 2} \cong \mathbb{F}_\ell^\times/\mathbb{F}_\ell^{\times 2}$ if $\ell$ is odd, and $\Z_2^\times / \Z_2^{\times 2} \cong (\Z/8\Z)^\times$). 
With powers of $\ell$ we can realize every integral valuation (recall that $d$ is an element of $\Z_{\ell}$), and the integers coprime to $\ell$ represent all elements of $\mathbb{Z}_\ell^\times/\mathbb{Z}_\ell^{\times 2}$, thus there is an integer representative.
Now suppose $\ell=2$ and consider a Cartan subgroup with parameters $(1,d)$ where $d$ is odd: in Proposition \ref{prop:ConditionsCD2} we show that the quadratic ring is $\mathbb{Z}_2[\zeta_6]$ and hence we can take as parameters $(1,-1)$. 
\end{rem}

\begin{prop}[Classification of Cartan subgroups for $\ell$ odd]\label{prop:ConditionsCDodd}
Suppose that $\ell$ is odd, and consider a Cartan subgroup of $\GL_2(\Z_\ell)$  with parameters $(0,d)$. It is
maximal if and only if $v_\ell(d)\leqslant 1$. It is unramified if and only if $\ell \nmid d$: it is then split if $d$ is a square in $\Z_\ell^\times$, and nonsplit otherwise.
\end{prop}
\begin{proof}
If $v_\ell(d)=1$ then $v_\ell(\omega)=1/2$ and hence $F$ is  a ramified extension of $\mathbb{Q}_\ell$.  If $v_\ell(d) \geqslant 2$ then $C$ is not maximal because $(\omega/ \ell)^{2}={d}/ \ell^{2} \in \Z_\ell$ and hence ${\omega}/{\ell}$ is in $\Z_{\ell}$. If $v_\ell(d) \leqslant 1$ then $R$ is a maximal order.  Indeed, let $R'$ be an order in $F$ containing $R$ and choose a $\Z_\ell$-basis $(1,\omega_1)$ of $R'$ satisfying $\omega_1^2=d_1 \in \Z_\ell$: writing $\omega=a \omega_1+b$ for some $a,b \in \Z_\ell$, we have $$d=\omega^2=(a^2 d_1+b^2)\cdot 1+(2ab)\cdot \omega_1$$ which implies $b=0$, thus $v_\ell(d)=2v_\ell(a)+v_\ell(d_1)$ and hence $v_\ell(a)=0$ and $R'=R$.

Now suppose $v_\ell(d)=0$.
 If $d$ is not a square, then $F=\mathbb{Q}_\ell(\sqrt{d})$ is an unramified extension of $\mathbb{Q}_\ell$ while if $d$ is a square the map $a + b \omega \mapsto (a + b \sqrt{d},a - b \sqrt{d})$ identifies $R$ and $\mathbb{Z}_\ell^2$.
\end{proof}

\begin{prop}[Classification of Cartan subgroups for $\ell=2$]\label{prop:ConditionsCD2}
Suppose that $\ell=2$, and consider a Cartan subgroup of $\GL_2(\Z_\ell)$  with parameters $(c,d)$.
 It is unramified if and only if $c=1$: it is then split for $d=0$ and nonsplit for $d$ odd. It is maximal and ramified if and only if $c=0$ and either $v_2(d)=1$ or $v_2(d)=0$ and $d \equiv 3 \pmod 4$.
\end{prop}
\begin{proof}
We keep the notation of the previous proof.  If $c=1$ and $d=0$ then $a+b\omega \mapsto (a,a+b)$ is an isomorphism between $R$ and $\Z_2^2$, so that $C$ is split. If $c=1$ and $d$ is odd then up to isomorphism we may suppose  $\omega^2=\omega+d$. This equation is separable over $\Z/2\Z$,  so $R$ is contained in the unique unramified quadratic extension of $\mathbb{Q}_2$, which is $\mathbb{Q}_2(\zeta_6)$. Since $2$ is inert, we will have shown that $C$ is nonsplit once we prove $R=\mathbb{Z}_2[\zeta_6]$. To show $\zeta_6\in R$, we write $\omega=a+b\zeta_6$ (with $a,b \in \mathbb{Z}_2$) and prove that $b$ is a unit: the equation $\omega^2=\omega+d$ gives $$(a^2-b^2)+\zeta_6(b^2+2ab)=(a+d)+\zeta_6b,$$ so $a+d$ has the same parity as $a^2-b^2$ and we deduce that $b$ is odd.

Conversely, if $C$ is unramified then no $\Z_2$-basis $(1,\omega)$ of $R$ satisfies $\omega^2 \in \Z_2$ and we must have $c=1$. 
Since $t^2=t+d$ has no solutions in $\Z_2^2$ for $d$ odd while it has solutions in $\mathbb{Q}_2$ if $d=0$, we deduce that $d=0$ (resp. $d$ is odd) for a split (resp. nonsplit) Cartan subgroup.
If $c=1$ we have seen that $C$ is maximal, so suppose $c=0$ and hence $\omega^2=d \in \Z_2$. Analogously to the previous remark we have $v_2(d) \leqslant 1$ if $C$ is maximal. 
By Remark \ref{rmk:IntegralParameters} we only need to consider those $d$ in $(\mathbb{Z}_2 \setminus\{0\})/ \mathbb{Z}_2^{\times 2} \cong (\mathbb{N} \times \mathbb{Z}_2^\times)/ \mathbb{Z}_2^{\times 2}$ with valuation $0$ or $1$, namely
$d=1,3,5,7$ and $d=2, 6, 10, 14$. We may conclude because it is known whether $\Z_2[\sqrt{d}]$ has index $1$ or $2$ in the ring of integers of $\Q_2(\sqrt{d})$, where for $d=1$ we set $\sqrt{d}=(1,-1)\in \mathbb{Z}_2^2$.
\end{proof}

\subsection{A concrete description of Cartan subgroups}
Let $(c,d)$ be parameters as in Theorem \ref{quadratic-rings}; we can then give a precise description for $C_R:=\operatorname{Res}_{R/\Z_\ell}(\mathbb{G}_m) \subset \GL_{2, \mathbb{Z}_\ell}$ as follows. For every $\mathbb{Z}_\ell$-algebra $A$, the $A$-points of $C_R$ are the subgroup of $\GL_{2, \mathbb{Z}_\ell}(A)$ given by:
\begin{equation*}
C_R(A)=\left\{ \begin{pmatrix}
x & dy \\ y & x+yc 
\end{pmatrix} : x,y \in A,\; \det \begin{pmatrix}
x & dy \\ y & x+yc 
\end{pmatrix} \in A^\times \right\}\,.
\end{equation*}

In particular the Cartan subgroup $C=C_R(\Z_\ell) $ is the set
\begin{equation}\label{nfaa}
C=\left\{ \begin{pmatrix}
x & dy \\ y & x+yc 
\end{pmatrix} : x,y \in \Z_{\ell},\; v_{\ell}(x(x+yc)-dy^2)=0 \right\}\,.
\end{equation}

\begin{rem}[Diagonal model for a split Cartan subgroup]\label{rmk:DiagonalModelWorks}
For the parameters $(c,d)$ of a split Cartan subgroup $C$ we have shown: if $\ell$ is odd, we have $c=0$ and $d$ is a square in $\Z_\ell^\times$;  if $\ell=2$, we have $(c,d)=(1,0)$. We deduce the existence of an isomorphism between $C$ and the group of diagonal matrices in $\GL_2(\mathbb Z_{\ell})$:
\begin{equation}\label{eq:SplitCartan}
\left\{ \begin{pmatrix}
X& 0 \\ 0 & Y 
\end{pmatrix} :\, X,Y \in \mathbb Z_{\ell}^\times
\right\} \, .
\end{equation}
We can define such an isomorphism for $\ell$ odd and for $\ell=2$ respectively as:
\begin{equation}\label{phil}
\varphi_{\ell} : \begin{pmatrix}
x & dy \\ y & x 
\end{pmatrix} \mapsto \begin{pmatrix}
x-y\sqrt{d} & 0 \\ 0 & x+y \sqrt{d}
\end{pmatrix}
\qquad 
\varphi_2 : \begin{pmatrix}
x & 0 \\ y & x+y 
\end{pmatrix} \mapsto \begin{pmatrix}
x & 0 \\ 0 & x+ y
\end{pmatrix}\,.
\end{equation}
We have $\det(\varphi_\ell(M)-I)=\det(M-I)$ and for any $n\geqslant 1$ we have $\varphi_\ell(M)-I\equiv 0 \bmod (\ell^n)$ if and only if $M-I\equiv 0 \bmod (\ell^n)$.
\end{rem}

{\bf Notation.} For a subset $X$ of $\GL_2(\Z_\ell)$ we denote by $X(n)$ the image of $X$ in $\GL_2(\Z/\ell^n\Z)$.

\begin{lem}\label{lemma:CardinalityCartans}
If $C$ is a Cartan subgroup of $\GL_2(\mathbb Z_{\ell})$ we have
$$\# C(1)= \left\{ 
\begin{array}{ll}
(\ell-1)^2 & \text{if $C$ is split}\\
(\ell-1)\cdot (\ell+1) & \text{if $C$ is nonsplit}\\
(\ell-1)\cdot \ell & \text{if $C$ is ramified}\\
\end{array}\right.$$
and for any $n\geqslant 1$ we have $\# C(n)=\# C(1) \cdot  \ell^{2n-2}$. 
\end{lem}
\begin{proof}
The assertion for
 $n=1$ is a straightforward computation, while for $n>1$ it follows from the (higher-dimensional version of) Hensel's Lemma \cite[Proposition 7.8]{Nekovar} because the Zariski closure of $C$ in  $\GL_{2,\mathbb Z_{\ell}}$ is smooth of relative dimension $2$.
\end{proof}

\subsection{Normalizers of Cartan subgroups}

\begin{lem}\label{lem-Norm}
A Cartan subgroup of $\GL_2(\Z_\ell)$ has index 2 in its normalizer. If $C$ is as in \eqref{nfaa}, its normalizer $N$ in $\GL_2(\Z_\ell)$ is the disjoint union of $C$ and  
$$C':=\begin{pmatrix}
1 & c \\ 0 & -1
\end{pmatrix} \cdot C\,.$$ 
We have instead $C':=\begin{pmatrix}
0 & 1 \\ 1 & 0
\end{pmatrix}\cdot C$ for a split Cartan subgroup as in \eqref{eq:SplitCartan}.
\end{lem}
\begin{proof}
An easy verification shows $C'\subset N$. If $A\in N$, there exist $x, y \in \Z_\ell$ such that we have 
\begin{equation*}A \begin{pmatrix}
0 & d \\ 1 & c
\end{pmatrix} A^{-1} = \begin{pmatrix}
x & yd \\ y & x+yc
\end{pmatrix}.
\end{equation*}
If $c=0$, by comparing traces we find $x=0$ and hence by comparing determinants  we have $(x,y) = (0,\pm 1)$. If $\ell=2$ and $c=1$, by comparing traces we find $y=1-2x$ and hence by comparing determinants we have $-x^2+x=0$, so $(x,y)$ is either $(0,1)$ or $(1,-1)$. We compute
\begin{equation*}A \begin{pmatrix}
0 & d \\ 1 & c
\end{pmatrix} = \begin{pmatrix}
x & yd \\ y & x+yc
\end{pmatrix} A
\end{equation*}
for any explicit value of $(x,y)$ as above, finding in each case $A \in C \cup C'$.
The last assertion about a split Cartan subgroup is well-known and easy to prove.
\end{proof}

\begin{rem}
If one considers a Cartan subgroup of $\GL_2(\Z_\ell)$  as the $\Z_\ell$-valued points of a maximal torus of $\GL_2$, the previous lemma also follows from the fact that any maximal torus in $\GL_2$ has index 2 in its normalizer (the Weyl group of $\GL_2$ is $\Z/2\Z$).\end{rem}

\begin{lem}\label{lemma:NormalizerAIsZero}
If $C$ is as in \eqref{nfaa} and $N$ is its normalizer then we have
\begin{equation}\label{eqcomplement}
N\setminus C =\Big\{\begin{pmatrix}
z & -dw+cz \\ w & -z
\end{pmatrix} \quad \text{
with $z,w\in \Z_{\ell}$\; and\; $v_{\ell}(-z^2+dw^2-czw)=0$\Big\}\,.}
\end{equation}
Consider $M \in N \setminus C$. If $\ell$ is odd, we have $M \not \equiv I \pmod{\ell}$; if $\ell=2$, we have $M \not \equiv I \pmod{4}$, and if $C$ is unramified we also have $M \not \equiv I \pmod{2}$.
\end{lem}
\begin{proof}
The first assertion follows from the previous lemma and \eqref{nfaa}. 
Since $M$ has trace zero, we have $M \not \equiv I \pmod{\ell}$ for $\ell$ odd and  $M \not \equiv I \pmod{4}$ for $\ell=2$. If $\ell=2$ and $C$ is unramified we know $c=1$ thus $M \equiv I \pmod{2}$ is impossible.
\end{proof}

\begin{rem}\label{nor}
By comparing \eqref{nfaa} and \eqref{eqcomplement}, we see that the sets $C(n)$ and $(N\setminus C) (n)$ are disjoint for $n\geqslant 2$ (if $\ell$ is odd or $C$ is unramified, for $n\geqslant 1$). By Lemma \ref{lem-Norm} we then have $\# N(n)=2 \cdot \# C(n)$.
\end{rem}

\subsection{The tangent space of a Cartan subgroup}\label{subsec-Lie}

Let $G$ be an open subgroup of either $\GL_2(\Z_\ell)$ or of the normalizer of a Cartan subgroup of $\GL_2(\Z_\ell)$. Then $G$ is an $\ell$-adic manifold, and there is a well-defined notion of a tangent space $T_{I}G$ at the identity. This is a $\Q_\ell$-vector subspace of $\operatorname{Mat}_2(\Q_\ell)$, of dimension equal to the dimension of $G$ as a manifold (in particular, if $G=\GL_2(\Z_\ell)$ we have $T_{I}G =\operatorname{Mat}_2(\Q_\ell)$). For our lifting questions, however, we are more interested in the `mod-$\ell$' tangent space, which can be defined either as the reduction modulo $\ell$ of the intersection of $T_{I}G$ with $\operatorname{Mat}_2(\Z_\ell)$, or as the tangent space to the modulo-$\ell$ fiber of the Zariski closure of $G$ in $\GL_{2,\Z_\ell}$. More concretely, the next two definitions describe the tangent space explicitly:

\begin{defi} If $C$ is as in \eqref{nfaa}, its tangent space is 
\[
\mathbb{T}:=\Big\{\begin{pmatrix}
x & dy \\ y & x+cy
\end{pmatrix} \, : \, x,y \in \Z/\ell \Z \Big\}
\]
where $(c,d)$ are here the reductions modulo $\ell$ of the parameters of $C$.
Write $\mathbb{T}^\times=C(1)$ for the subset of $\mathbb{T}$ consisting of the invertible matrices.
\end{defi}

We clearly have $\#\mathbb{T}=\ell^2$ and by Lemma \ref{lemma:CardinalityCartans} we also know $\#\mathbb{T}^\times$. So we get:
\begin{equation}\label{eqprop:XB}
\begin{tabular}{|c|c|c|c|}
\hline
{Type of $C$}  &$\#\mathbb{T}$ & $\#\mathbb{T}^\times$ & $\#\mathbb{T}-\#\mathbb{T}^\times-1$ \\ \hline
split & $\ell^2$ & $(\ell-1)^2$ & $2(\ell-1)$ \\ \hline
nonsplit & $\ell^2$ & $\ell^2-1$ & $0$ \\ \hline
ramified & $\ell^2$ & $\ell(\ell-1)$ & $\ell-1$ 
\\ \hline
\end{tabular}
\end{equation}

We define the tangent space of an open subgroup of the normalizer of $C$ as the tangent space of $C$. We also define the tangent space of an open subgroup of $\GL_2(\Z_\ell)$ as follows:

\begin{defi}\label{LieGL}
Let $G$ be an open subgroup of $\GL_2(\Z_\ell)$. The tangent space of $G$ is $\mathbb{T}:=\operatorname{Mat}_2(\Z/\ell \Z)$ and we write $\mathbb{T}^\times=\GL_2(\Z/\ell \Z)$.
\end{defi}

For $\GL_2(\Z_\ell)$ we have  $\#\mathbb{T}=\ell^4$, $\#\mathbb{T^\times}=\ell (\ell-1)^2(\ell+1)$ and $\#\mathbb{T}-\#\mathbb{T}^\times-1=(\ell+1)(\ell^2-1)$.

\section{The group structure of the $1$-Eigenspace}\label{sec-1ES}

\subsection{The level} Let $G'$ be either $\GL_2(\mathbb{Z}_\ell)$ or the normalizer of a Cartan subgroup of $\GL_2(\mathbb{Z}_\ell)$. Let $G$ be an open subgroup of $G'$ of finite index $[G':G]$. Call $G'(n)$ and $G(n)$ the reductions of $G'$ and $G$ modulo $\ell^n$, that is their respective images in $\GL_2(\Z/\ell^n\Z)$.

If $n$ is the smallest positive integer such that we have $[G'(n):G(n)]=[G':G]$, we  define the \emph{level} $n_0$ of $G$ as
\[
n_0 = \begin{cases} \begin{array}{ll}
\max\{n,2\} & \text{if } \ell = 2 \text{ and $G'$ is the normalizer of a ramified Cartan} \\ n & \text{otherwise.} \end{array} \end{cases}
\]

\begin{rem}\label{rem:Level}
It is easy to check that all our statements involving the notion of \emph{level} remain true if $n_0$ is replaced by any larger integer. 
\end{rem}

All matrices in $G'$ that are congruent to the identity modulo $\ell^{n_0}$ belong to $G$. In other words, $G$ is the inverse image of $G{(n_0)}$ for the reduction map $G' \to G'(n_0)$.
Consequently we have 
\begin{equation}\label{n0-forH}
[G'(n):G(n)]=[G':G]\qquad \text{for every $n\geqslant n_0$\,.}
\end{equation}
The dimension of $G'$ is $4$ if $G'=\GL_2(\Z_\ell)$ and is $2$ otherwise, and we have
\begin{equation}\label{dimG}
[G(n+1):G(n)]=[G'(n+1):G'(n)]=\ell^{\dim G'}\qquad \text{for every $n\geqslant n_0$\,.}
\end{equation}

\begin{rem}\label{kerT}
Let $G$ be an open subgroup of either $\GL_2(\Z_\ell)$ or the normalizer of a  Cartan subgroup of $\GL_2(\Z_\ell)$. Let $n_0$ be the level of $G$.
For any $n\geqslant n_0$ the map $M \mapsto \ell^{-n}(M-I)$ identifies the tangent space of $G$ with the kernel of $G(n+1) \to G(n)$. This is immediate for $\GL_2(\Z_\ell)$, and for Cartan subgroups it follows from \eqref{nfaa}. The assertion also holds for normalizers of Cartan subgroups  because by Lemma \ref{lemma:NormalizerAIsZero} all matrices reducing to the identity in $G(n)$ are contained in the Cartan subgroup.
\end{rem}

\subsection{The $1$-Eigenspace}

We identify an element of $\GL_2(\mathbb Z_{\ell})$ with an automorphism of the direct limit in $n$ of $(\Z/\ell^n \Z)^2$. For all integers $a,b\geqslant 0$, if $X\subseteq \GL_2(\mathbb Z_{\ell})$ we define
\begin{equation}
\mathcal M_{a,b}(X):=\{M\in X: \, \ker (M-I) \simeq  \Z/{\ell^a}\Z\times \Z/{\ell^{a+b}}\Z\}
\end{equation}
and call $\mathcal M_{a,b}(X;n)$ the reduction of $\mathcal M_{a,b}(X)$ modulo $\ell^n$. To ease notation, we write $\mathcal M_{a,b}:=\mathcal M_{a,b}(G)$ and $\mathcal M_{a,b}(n):=\mathcal M_{a,b}(G;n)$.

We consider the normalized Haar measure on $G$ and call $\mu_{a,b}$ the measure of the set $\mathcal M_{a,b}$. Since $\mathcal M_{a,b}(n)$ is a subset of $G(n)$, we may consider its measure $$\mu_{a,b}(n):=\#\mathcal M_{a,b}(n)/\#G(n)\,.$$
The sets $\mathcal{M}_{a,b}$ are pairwise disjoint, but the same is not necessarily true for the sets $\mathcal{M}_{a,b}(n)$. 
The sequence $\mu_{a,b}(n)$ is constant for $n>a+b$: this shows that $\mu_{a,b}$ is well-defined and that we have $\mu_{a,b}=\mu_{a,b}(n)$ for every $n>a+b$.

\begin{rem}\label{vuoto}
We clearly have $\mathcal M_{a,b}=G\cap \mathcal M_{a,b}(G')$. Moreover, we have
$$\mathcal M_{a,b}=\emptyset \qquad \Leftrightarrow\qquad  G(n_0)\cap \mathcal M_{a,b}(G'; n_0)=\emptyset\,.$$
Indeed we know $\mathcal M_{a,b}(n_0)\subseteq G(n_0)\cap \mathcal M_{a,b}(G';n_0)$ so if the latter is empty so is $\mathcal{M}_{a,b}$.
Conversely, matrices in $\mathcal M_{a,b}(G')$ whose reduction modulo $\ell^{n_0}$ lies in $G(n_0)$ are in $\mathcal M_{a,b}$ because $G$ is the inverse image in $G'$ of $G(n_0)$.
\end{rem}

\subsection{Additional notation}
We write $\det_{\ell}$ for the \emph{$\ell$-adic valuation of the determinant}. If $M$ is in $\operatorname{Mat}_2(\Z/\ell^n\Z)$, then $\det(M)$ is well-defined modulo $\ell^n$ so we can write $\det_{\ell}(M)\geqslant n$ if the determinant is zero modulo $\ell^n$. Notice that the matrices in $\operatorname{Mat}_2(\Z_{\ell})$ that are zero modulo $\ell^a$ for some $a\geqslant 0$ and with a given reduction modulo $\ell^n$ for some $n>a$ have a determinant which is well-defined modulo $\ell^{a+n}$.
More generally, if $p$ is a polynomial with integer coefficients and $z_1, z_2$ are in $\mathbb{Z}/\ell^n\mathbb{Z}$ then we write $v_\ell(p(z_1,z_2))$ for the minimum of $v_\ell(p(Z_1,Z_2))$ over all lifts $Z_1, Z_2$ of $z_1, z_2$ to $\mathbb{Z}_\ell$. For example, if $z \equiv \ell^{t} \pmod{\ell^n}$ with $t<n$ then we have $v_\ell(z^2)=2t$ because all lifts $Z$ of $z$ satisfy $v_\ell(Z^2)=2t$.

\subsection{Conditions related to the group structure of the $1$-Eigenspace}

\begin{lem}\label{condi-det}
The set $\mathcal M_{a,b}$ consists of the matrices $M\in G$ that satisfy
\begin{equation}\label{conditions-Mab}
M-I \equiv 0\!\!\!\! \pmod{\ell^a},\; M-I \not \equiv 0 \!\!\!\!\pmod{\ell^{a+1}}\quad\text{and}\quad \mathrm{det}_{\ell}(M-I)=2a+b\,.
\end{equation}
For every $n>a+b$ the set $\mathcal M_{a,b}$ is the preimage of $\mathcal M_{a,b}(n)$ in $G$, and $\mathcal{M}_{a,b}(n)$ consists of the matrices $M\in G(n)$ satisfying \eqref{conditions-Mab}.
\end{lem}
\begin{proof}
The necessity of \eqref{conditions-Mab} follows from the fact that for $A\in \operatorname{Mat}_2(\Z_\ell)$  the order of the kernel of $A$ (considered as acting on the direct limit $\varinjlim_n (\Z/\ell^n\Z)^2$) equals $\ell^{\det_\ell A}$, that is, there are $\ell^{\det_\ell A}$ points $x$ in $\varinjlim_n (\Z/\ell^n\Z)^2$ such that $Ax=0$.
 Now suppose that $M\in G$ satisfies \eqref{conditions-Mab}, and write $M=I+\ell^a A$ for some $A\in \operatorname{Mat}_2(\Z_{\ell})$ which is non-zero modulo $\ell$. We have $\det_\ell(A)=b$. Since $A$ is nonzero modulo $\ell$, the kernel of $A$ is cyclic. Thus $\ker(A)\simeq \Z/\ell^b\Z$ and hence $\ker(M-I)\simeq \Z/\ell^a \Z \times \Z/\ell^{a+b}\Z$. If $n>a+b$, \eqref{conditions-Mab} holds for the matrices in $\mathcal{M}_{a,b}(n)$ and their preimages in $G$.
\end{proof}

By Remark \ref{rmk:DiagonalModelWorks} and Lemma \ref{condi-det}, the maps in \eqref{phil}  preserve $\mathcal{M}_{a,b}$ and $\mathcal{M}_{a,b}(n)$, thus for a split Cartan subgroup we can indifferently use the general model \eqref{nfaa} or the diagonal model \eqref{eq:SplitCartan}.

\subsection{Existence of the Haar measure}

A fundamental tool in dealing with Haar measures on profinite groups is the following simple lemma:
\begin{lem}\label{lemma:Haar}{\cite[Lemma 18.1.1]{MR2445111}}
Let $\mathcal G_1$ be a profinite group equipped with its normalized Haar measure, and let $\mathcal G_2$ be an open normal subgroup of $\mathcal G_1$. Call $\pi$ the natural projection $\mathcal G_1 \to \mathcal G_1/\mathcal G_2$. For any subset $S$ of the finite group $\mathcal G_1/\mathcal G_2$, the set $\pi^{-1}(S)$ is measurable in $\mathcal G_1$, and its Haar measure is $\#S/\#(\mathcal G_1/\mathcal G_2)$.
\end{lem}

\begin{lem}\label{lemma:EverythingHasWellDefinedAB}
For all integers $a,b\geqslant 0$ the set $\mathcal{M}_{a,b}$ is measurable in $G$ and we have $\mu_{a,b}=\mu_{a,b}(n)$ whenever $n>a+b$. In particular we have $\mu_{a,b}=0$ if and only if $\mathcal{M}_{a,b}=\emptyset$. The set $\bigcup_{a,b \in \mathbb N} \mathcal{M}_{a,b}$ is measurable in $G$, and its complement has measure zero.
\end{lem}
\begin{proof}
For the first assertion apply Lemma \ref{lemma:Haar} to $G$, $\ker(G \to G(n))$ and $\mathcal{M}_{a,b}(n)$, noticing that $\mathcal{M}_{a,b}$ is the preimage of 
$\mathcal{M}_{a,b}(n)$ in $G$ by Lemma \ref{condi-det}. The set $\mathcal M :=\bigcup_{a,b \in \mathbb N} \mathcal{M}_{a,b}$ is measurable because it is a countable union of measurable sets. We now prove $\mu(G\setminus \mathcal M)=0$. Fix $n_0$ as in \eqref{n0-forH} and for $n\geqslant n_0$ call $\pi_n: G\rightarrow G(n)$ the reduction modulo $\ell^n$. We have $G\setminus \mathcal M \subseteq \pi_n^{-1}\left( \pi_n\left(G\setminus \mathcal M \right) \right)$, so by Lemma \ref{lemma:Haar} it suffices to show that
\begin{equation}\label{10}
\mu(\pi_n\left(G\setminus \mathcal M \right))=\frac{\# \pi_n\left( G\setminus \mathcal M \right)}{\#G(n)}
\end{equation}
tends to 0 as $n$ tends to infinity. By \eqref{dimG} we know that $\#G(n)$ is a constant times $\ell^{n \dim G'}$. Let $G'_{\infty}$ be the closed $\ell$-adic analytic subvariety of $G'$ defined by $\det(M-I)=0$.
We have $G\setminus \mathcal M\subseteq G'_{\infty}$ because for any $M\in G$ with $\det(M-I) \neq 0$ there exists $n$ such that $M \not \equiv I \pmod{\ell^n}$ and $\det_\ell(M-I) \leqslant n$, whence $M \in \mathcal{M}$. Thus the numerator in \eqref{10} is at most $\# \pi_n(G'_{\infty})$, which by \cite[Theorem 4]{MR656627} is at most a constant times $\ell^{n \dim(G'_{\infty})}=\ell^{n(\dim G' -1)}$.\end{proof}

\subsection{The complement of a Cartan subgroup in its normalizer}\label{star}

Fix a Cartan subgroup $C$ of $\GL_2(\Z_\ell)$ and denote by $N$ its normalizer. If $G$ is an open subgroup of $N$, set
\[
\mathcal{M}^*_{a,b}:=(N \setminus C) \cap \mathcal{M}_{a,b}\,.
\]

We denote by $\mathcal{M}^*_{a,b}(n)$ the reduction of $\mathcal{M}^*_{a,b}$ modulo $\ell^n$, that is its image in $G(n)$.

If $G$ is not contained in $C$, the sets $G \cap C$ and $G \cap (N \setminus C)$ are measurable and  have measure $1/2$ in $G$ because of Lemma \ref{lemma:Haar} applied to the canonical projection $G \to G/(G \cap C) \cong \Z/2\Z$. In particular we have
\[
\mu(\mathcal{M}_{a,b}) = \mu\left(\mathcal{M}_{a,b} \cap C\right) + \mu\left( \mathcal{M}^*_{a,b} \right).
\]
Since $\mu_N\left( \mathcal{M}_{a,b} \cap C \right)=1/2 \cdot \mu_C\left( \mathcal{M}_{a,b} \cap C \right)$, to determine $\mu_{a,b}$ we are reduced to computing $\mu( \mathcal{M}^*_{a,b})$ and  studying $G\cap C$, which is open in the Cartan subgroup $C$.

\begin{prop}\label{complement}
We have $\mathcal{M}^*_{a,b} = \emptyset$ for $a > 1$ (if $\ell$ is odd or $C$ is unramified, for $a>0$).
\end{prop}

\begin{proof} This is a consequence of Lemma \ref{lemma:NormalizerAIsZero}.
\end{proof}

\section{First results on the cardinality of $\mathcal{M}_{a,b}(n)$}

\begin{thm}\label{thm:GeneralLift}
Let $G'$ be either $\GL_2(\Z_\ell)$ or the normalizer of an unramified Cartan subgroup of $\GL_2(\Z_\ell)$. Let $G$ be an open subgroup of $G'$ of level $n_0$.
Call ${\mathcal{H}_{a,b}(n)}$ the set of matrices $M$ in $G(n)$ satisfying the following conditions:
\begin{enumerate}
\item if $a>0$, $M \equiv I \pmod{\ell^a}$; if $n>a$, $M \not \equiv I \pmod{\ell^{a+1}}$;
\item if $a<n \leqslant a+b$, $\det_\ell(M-I) \geqslant a+n$; if $n>a+b$, $\det_\ell(M-I)=2a+b$.
\end{enumerate}
For every integer $n\geqslant 1$ define 
\[
f(n)=\left\{\begin{array}{ll}
1 & \text{if $n < a$} \\
\#\mathbb{T}^\times & \text{if $n=a, b=0$} \\
\#\mathbb{T}-\#\mathbb{T}^\times-1 & \text{if $n=a, b>0$}\\
\#\mathbb{T}\cdot \ell^{-1}  & \text{if $a<n<a+b$} \\
\#\mathbb{T}\cdot (1-\ell^{-1}) & \text{if $n=a+b$, $b>0$} \\
\#\mathbb{T} & \text{if $n>a+b$}\,. \\
\end{array}\right.
\]

Then the following holds:
\begin{enumerate} 
\item[(i)] For every $n\geqslant n_0$ we have ${\# \mathcal{H}_{a,b}(n+1)}=f(n)\cdot {\#\mathcal{H}_{a,b}(n)}$. For each $M \in\mathcal{H}_{a,b}(n)$ the number of matrices ${M'} \in\mathcal{H}_{a,b}(n+1)$ such that ${M'} \equiv M \pmod{\ell^n}$ equals $f(n)$.
\item[(ii)] If $\mathcal{M}_{a,b}\neq \emptyset$ and $n \geqslant n_0$, or if $n>a+b$, we have $\mathcal{M}_{a,b}(n)=\mathcal{H}_{a,b}(n)$.
\item[(iii)] For $a\geqslant n_0$ we have $\mathcal{M}_{a,b}= \emptyset$ if and only if $b>0$ and $\#\mathbb{T}-\#\mathbb{T}^\times-1=0$.
\end{enumerate}
\end{thm}
\begin{proof} We first prove (i). Since $n\geqslant n_0$, all lifts to $G'(n+1)$ of matrices in ${\mathcal{H}_{a,b}(n)}$ are in $G(n+1)$. If $n>a+b$ then clearly every lift to $G(n+1)$ of a matrix in $\mathcal{H}_{a,b}(n)$ belongs to $\mathcal{H}_{a,b}(n+1)$. If $n<a$, the sets ${\mathcal{H}_{a,b}(n)}$ and ${\mathcal{H}_{a,b}(n+1)}$ contain only the identity and we are done.

Suppose $n=a$: the only matrix in ${\mathcal{H}_{a,b}(a)}$ is the identity, so we apply Remark \ref{kerT}. For $b=0$ we count the matrices of the form $I+\ell^a T$ with $T \in \mathbb{T}$ and $\det_{\ell}(T)= 0$. For $b>0$ we count those $M\in \mathcal{H}_{a,b}(a+1)$ that are congruent to the identity modulo $\ell^{a}$ but not modulo $\ell^{a+1}$ and such that $\det_\ell(M-I) \geqslant 2a+1$: this means $M=I+\ell^a T$, where $T \in \mathbb{T}$ with $\det_\ell T \neq 0$, and excluding $T=0$.

Now consider the case $a<n<a+b$. Let $M \in\mathcal{H}_{a,b}(n)$ and fix some lift $L$ to $G(n+1)$.  
The lifts of $M$ are those matrices of the form $M' =L+\ell^n T$ with $T \in \mathbb{T}$, unless $G'$ is the normalizer of a Cartan subgroup $C$ and $M\notin C(n)$, for which by Lemma \ref{lemma:NormalizerAIsZero} we have $a=0$ and $T\in \mathbb T_1$, where 
\begin{equation}\label{defT}
\mathbb T_1:=\Big\{\begin{pmatrix}
z & -dw+cz \\ w & -z
\end{pmatrix}\quad \text{where $z,w\in \Z/\ell\Z$}\Big\}\,.
\end{equation}
Write $L-I=\ell^a N$ for some $N\in \operatorname{Mat}_2(\Z/\ell^{n+1-a}\Z)$. Since $b>n-a$, we have $\det(N)=\ell^{n-a} z$ for some $z \in \Z/\ell\Z$.
Setting $(N \bmod \ell)=(n_{ij})$ and $T=(t_{ij})$, we get the following congruence modulo $\ell^{n+1-a}$: 
\[
\det (M'-I ) \equiv \ell^{2a}\cdot \det (N+\ell^{n-a}T) 
\equiv
\ell^{n+a} \left(z + n_{11}t_{22}+ n_{22}t_{11}-n_{21}t_{12}-n_{12}t_{21} \right). 
\]
So the condition for $M'$ to be in ${\mathcal{H}_{a,b}(n+1)}$ is  
\begin{equation}\label{eq:Lifts1}
z + n_{22}t_{11}+n_{11}t_{22}-n_{21}t_{12}-n_{12}t_{21}= 0.
\end{equation}
We conclude by checking that this equation defines an affine subspace of codimension 1 in $\mathbb{T}$ (resp. in $\mathbb{T}_1$, if $M \not \in C(n)$). The equation is nontrivial because at least one of the $n_{ij}$ is nonzero, and this remark suffices for  $\GL_2(\Z_{\ell})$. If $G'$ is the normalizer of a Cartan subgroup, we also have to check that \eqref{eq:Lifts1} is independent from the equations defining $\mathbb{T}$ (resp.  $\mathbb{T}_1$), which are 
\[
\begin{cases}
t_{22}=t_{11}+ct_{21} \\
t_{12}=dt_{21}.
\end{cases}
\qquad\text{resp.}\qquad
\begin{cases} t_{22}=-t_{11} \\ t_{12}=-dt_{21}+ct_{11}. \end{cases}
\]

For the elements of $\mathbb{T}$, noticing that $(N \bmod \ell )$ depends only on $(M \bmod \ell^{a+1})$ and it is in $\mathbb T\setminus \{0\}$, 
we can rewrite \eqref{eq:Lifts1} as  
\begin{equation*}
z + (2n_{11}+cn_{21})t_{11}+(n_{11}c-2dn_{21})t_{21}=0
\end{equation*}
and we can easily check by Proposition \ref{prop:ConditionsCD2} that  $2n_{11}+cn_{21}$ or $n_{11}c-2dn_{21}$ is nonzero.

For the elements of $\mathbb{T}_1$, we again conclude by Proposition \ref{prop:ConditionsCD2} because $a=0$ and we have $(N+I \bmod \ell )\in \mathbb{T}_1\setminus\{0\}$, thus \eqref{eq:Lifts1} becomes
\begin{equation}\label{eq:LiftsNormalizer}
z - (2(n_{11}+1)+c n_{21})t_{11}+ (2dn_{21}-c(n_{11}+1))t_{21}=0. 
\end{equation}

If $n=a+b$ and $b>0$, we can reason as in the previous case. Now the condition for $M'$ to be in ${\mathcal{H}_{a,b}(a+b+1)}$ is that \eqref{eq:Lifts1} is \textit{not} satisfied: we conclude because that equation has $\ell^{-1}\cdot \#\mathbb{T}$ solutions.

We now prove (ii). The assertion for $n>a+b$ is the content of Lemma \ref{condi-det}, so in particular we know $\mathcal{M}_{a,b}(a+b+1)=\mathcal{H}_{a,b}(a+b+1)$ and we may suppose $n \leqslant a+b$. We clearly have $\mathcal{M}_{a,b}(n) \subseteq\mathcal{H}_{a,b}(n)$ and are left to prove the other inclusion.
The assumption $\mathcal{M}_{a,b}\neq \emptyset$ implies that for all $x\geqslant 1$ the sets $\mathcal{M}_{a,b}(x)$ and $\mathcal{H}_{a,b}(x)$ are non-empty and hence   by (i) we know $f(x)\neq 0$ for all $x \geqslant n_0$.
Thus for any $M \in\mathcal{H}_{a,b}(n)$ there is some ${M'} \in {\mathcal{H}_{a,b}(a+b+1)}$ satisfying ${M'} \equiv M \pmod{\ell^n}$, and we deduce $M\in \mathcal{M}_{a,b}(n)$.

Finally, we prove (iii). The condition $f(a)=0$ is equivalent to $b>0$ and $\#\mathbb{T}-\#\mathbb{T}^\times-1=0$. By (i), if $f(a)=0$ then $\mathcal{H}_{a,b}(a+1)$ is empty and hence also $\mathcal{M}_{a,b}(a+1)$ and $\mathcal{M}_{a,b}$ are empty. If $f(a)\neq 0$ then we have $f(x)\neq 0$ for all $x\geqslant 1$. Since $\mathcal{H}_{a,b}(a)$ contains the identity, we deduce that $\mathcal{H}_{a,b}(a+b+1)=\mathcal{M}_{a,b}(a+b+1)$ is nonempty, and hence $\mathcal{M}_{a,b}\neq \emptyset$.
\end{proof}

\section{The number of lifts for the reductions of matrices}

\subsection{Main result}

We study the lifts of a matrix $M\in \mathcal{M}_{a,b}(n)$ to $\mathcal{M}_{a,b}(n+1)$, namely the matrices in $\mathcal{M}_{a,b}(n+1)$ which are congruent to $M$ modulo $\ell^n$.

\begin{thm}\label{thm-lifts}
Let $G$ be an open subgroup of either $\GL_2(\Z_\ell)$ or the normalizer $N$ of a Cartan subgroup $C$ of $\GL_2(\Z_\ell)$. Let $n_0$ be the level of $G$.
For $n\geqslant n_0$ the number of lifts of a matrix $M\in \mathcal{M}_{a,b}(n)$ to $\mathcal{M}_{a,b}(n+1)$ is independent of $M$ in the first case, while in the second case it depends at most on whether $M$ belongs to either $C(n)$ or $(N\setminus C)(n)$.
\end{thm}
\begin{proof}[Proof of Theorem \ref{thm-lifts}]
If $G$ is open in $\operatorname{GL}_2(\Z_\ell)$ or if $C$ is unramified, the number of lifts of $M$ to $G(n+1)$ is independent of $M$ by Theorem \ref{thm:GeneralLift} (i). 
If $C$ is ramified the assertion follows from Theorems \ref{thm:LiftsRamifiedCartan} and  \ref{thm:LiftsNormalizer}.
\end{proof}

\begin{exa}
The number of lifts may indeed depend on the coset of $N/C$. Suppose that $\ell$ is odd and consider the Cartan subgroup $C$ of $\operatorname{GL}_2(\Z_\ell)$ with parameters $(0,\ell)$. If $G$ is the normalizer of $C$ then the matrices 
\[
\begin{pmatrix}
1 & \ell \\ 1 & 1
\end{pmatrix}
\qquad\text{and}\qquad
\begin{pmatrix}
1 & -\ell \\ 1 & -1
\end{pmatrix}
\]
are in $\mathcal M_{0,1}$ and their reductions modulo $\ell$ have respectively $\ell^2$ and $\ell^2-\ell$ lifts to $\mathcal M_{0,1}(2)$. Indeed, their lifts to $G(2)$ are of the form 
$$L=\begin{pmatrix}
1+\ell u & \ell \\ 1+\ell v & 1+\ell u
\end{pmatrix}\qquad \text{and}\qquad
L'=\begin{pmatrix}
1+\ell u & -\ell \\ 1+\ell v & -1-\ell u
\end{pmatrix}
$$
respectively, where $u,v\in \Z/\ell\Z$: we have 
$\det_\ell (L-I)=1$ for every $u,v$ while $\det_\ell (L'-I)=1$ holds if and only if $2u-1 \not \equiv 0 \pmod \ell$.
\end{exa}

\subsection{Ramified Cartan subgroups}

\begin{thm}\label{thm:LiftsRamifiedCartan}
Let $G$ be open in a ramified Cartan subgroup of $\GL_2(\Z_\ell)$. 
Let $n_0$ be the level of $G$. For all $a,b\geqslant 0$ and for all $n\geqslant n_0$ the number of lifts of a matrix $M\in \mathcal{M}_{a,b}(n)$ to $\mathcal{M}_{a,b}(n+1)$ is independent of $M$.
\end{thm}

\begin{proof}
For $n \leqslant a$ the set $\mathcal{M}_{a,b}(n)$ consists at most of the identity matrix, so suppose $n>a$. Let $(0,d)$ be the parameters for the Cartan subgroup (for convenience, we do not use a different notation for $d$ and its reductions  modulo powers of $\ell$). The matrices in $\mathcal{M}_{a,b}(n)$ are of the form  
\begin{equation}\label{emmmmmm}
 M=I+\ell^a\begin{pmatrix}
x & dy \\ y & x
\end{pmatrix}
\end{equation}
where $x,y\in \Z/\ell^{n-a}\Z$ are not both divisible by $\ell$ and have lifts $X,Y\in \Z_\ell$ satisfying $v_\ell(X^2-dY^2)=b$.

\emph{If all matrices in $\mathcal{M}_{a,b}(n)$ satisfy $x \equiv 0 \pmod{\ell^{n-a}}$ then they all have the same number of lifts to $\mathcal{M}_{a,b}(n+1)$.} Since $y$ is a unit, for any $M_1, M_2 \in \mathcal{M}_{a,b}(n)$ there is an obvious bijection between the lifts of $M_1-I$ and of $M_2-I$ given by rescaling by a suitable unit.

\emph{If some $M_0\in \mathcal{M}_{a,b}(n)$ satisfies $x_0 \not \equiv 0 \pmod{\ell^{n-a}}$ and either $v_\ell(x_0^2)\neq v_\ell(d)$ or $v_\ell(x_0^2)=v_\ell(d)=v_\ell(x_0^2-dy_0^2)$ then every matrix in $\mathcal{M}_{a,b}(n)$ has $\ell^2$ lifts to $\mathcal{M}_{a,b}(n+1)$.}

It suffices to show that for $M\in \mathcal{M}_{a,b}(n)$ all lifts $X,Y$ of $x,y$ to $\Z_{\ell}$  satisfy $v_\ell(X^2-dY^2)=b$ because this implies that all lifts of $M$ to $G$ belong to $\mathcal M_{a,b}$.

If $v_\ell(x_0^2)< v_\ell(d)$ then for any $X_0,Y_0$ lifting $x_0,y_0$ we have $v_\ell(X_0^2-dY_0^2)=v_\ell(X_0^2)=v_\ell(x_0^2)$, so this number is independent of the lift and it is equal to $b$. In particular, we have $b<v_\ell(d)$ and $b<2(n-a)$. For $M\in \mathcal{M}_{a,b}(n)$ there exist lifts $X,Y \in \mathbb{Z}_\ell$ of $x,y$ that satisfy $v_\ell(X^2-dY^2)=b$. We deduce $v_\ell(X^2)=b$ and hence $v_\ell(x^2)\leqslant b$: since $v_\ell(x^2)< v_\ell(d)$ and $x \not \equiv 0 \pmod{\ell^{n-a}}$ we can  reason as for $M_0$ and we conclude.

If $v_\ell(x_0^2)> v_\ell(d)$ then $y_0$ must be a unit, so we have $v_\ell(d)=v_\ell(x_0^2-dy_0^2)$, and the same holds for all lifts $X_0,Y_0$. In particular, we have $b=v_\ell(d)<2(n-a)$.

If $v_\ell(x_0^2)=v_\ell(d)=v_\ell(x_0^2-dy_0^2)$, we write $x_0=\ell^k u_0$ and $d=\ell^{2k} \delta$, where $u_0, \delta$ are units and $k<n-a$. Then $u_0^2-\delta y_0^2$ is a unit and hence $v_\ell(U_0^2-\delta Y_0^2)=0$ for all lifts $U_0, Y_0$ of $u_0, y_0$. We deduce $v_\ell(X_0^2-dY_0^2)= v_\ell(d)$ for all lifts $X_0,Y_0$ of $x_0,y_0$ and again we have $b = v_\ell(d)<2(n-a)$.

So suppose $b = v_\ell(d)<2(n-a)$. For $M\in \mathcal{M}_{a,b}(n)$ there are lifts $X,Y$ of $x,y$ satisfying $v_\ell(X^2-dY^2)=b$ and hence $v_\ell(X^2)\geqslant v_\ell(d)$ and $v_\ell(x^2-dy^2)\leqslant v_\ell(d)$.
If $x \not \equiv 0 \pmod{\ell^{n-a}}$ then either $v_\ell(x^2) \neq v_\ell(d)$ or we have $v_\ell(x^2)= v_\ell(d)$ and $v_\ell(x^2-dy^2)=v_\ell(d)$, so we can reason as for $M_0$. If $x \equiv 0 \pmod{\ell^{n-a}}$ then $v_\ell(x^2)>v_\ell(d)$ and $y$ is a unit: we deduce $v_\ell(X^2-dY^2)=b$ for all lifts $X,Y$.

\emph{If some $M_0 \in \mathcal{M}_{a,b}(n)$ satisfies $x_0 \not \equiv 0 \pmod{\ell^{n-a}}$, $v_\ell(x_0^2)=v_\ell(d)$ and $v_\ell(x_0^2-dy_0^2)>v_\ell(d)$, then no $M\in \mathcal{M}_{a,b}(n)$ has $x \equiv 0 \pmod{\ell^{n-a}}$.} From $v_\ell(x_0^2-dy_0^2)>v_\ell(d)$ we deduce $b>v_\ell(d)$. Supposing that such an $M$ exists, let $X,Y$ be lifts of $x, y$ to $\mathbb{Z}_\ell$ such that $v_\ell(X^2-dY^2)=b$. 
Since $y$ must be a unit, $v_\ell(X^2-dY^2)>v_\ell(d)$ implies $v_\ell(X^2)=v_\ell(dY^2)=v_\ell(d)$. We deduce $v_\ell(x_0)=v_\ell(X) \geqslant v_\ell(x)$, which contradicts $v_\ell(x_0) < n-a \leqslant v_\ell(x)$.

\emph{Finally, if all $M \in \mathcal{M}_{a,b}(n)$ satisfy $x \not \equiv 0 \pmod{\ell^{n-a}}$, $v_\ell(x^2)=v_\ell(d)$ and $v_\ell(x^2-dy^2)>v_\ell(d)$ then the number of lifts of $M$ to $\mathcal{M}_{a,b}(n+1)$ only depends on $G,d,n,a,b$.} The hypotheses imply $v_\ell(x^2)=v_\ell(dy^2)$ and hence $y$ is a unit, otherwise neither $x$ nor $y$ would be units. We can write 
$d=\ell^{2k} \delta$, $x=\ell^{k} u$ and $X = \ell^k U$ where  $\delta,u,U$ are units.
We are counting the reductions modulo $\ell^{n-a+1}$ of the pairs $(X,Y)\in \Z_\ell^2$ that satisfy:
\begin{equation}\label{syss}
\begin{cases}
U \equiv u \pmod{\ell^{n-a-k}} \\
Y \equiv y \pmod{\ell^{n-a}} \\
v_\ell(U^2- \delta Y^2)=b-2k.
\end{cases}
\end{equation}

{Consider the case where $\ell$ is odd.} If $b-2k\leqslant n-a-k$, the third condition of \eqref{syss} is a consequence of the first two because it only depends on $U,Y$ through $u,y$ (since by assumption it holds for {some} lifts, it then holds for {all} lifts). So $M$ has $\ell^2$ lifts to $\mathcal{M}_{a,b}(n+1)$. Now suppose that $b-2k> n-a-k$.

We know that $\delta$ is a square in $\Z_\ell^\times$ because $\ell \mid u^2-\delta y^2$ and $\ell \nmid y$. Since $\ell$ is odd, we may assume without loss of generality that $u-\sqrt{\delta}y\equiv 0 \pmod{\ell}$ and $u+\sqrt{\delta}y\not\equiv 0 \pmod{\ell}$.
We may then rewrite the third condition of \eqref{syss} as
\begin{equation}\label{syss2}
U-\sqrt{\delta}\, Y \equiv 0 \pmod{\ell^{b-2k}}, \qquad U-\sqrt{\delta}\, Y \not \equiv 0 \pmod{\ell^{b-2k+1}}.
\end{equation}
If we choose $(Y \bmod {\ell^{n-a+1}})$ arbitrarily among the lifts of $y$, \eqref{syss2} uniquely determines the value of $(U \bmod {\ell^{n-a-k+1}})$, so $M$ has $\ell$ lifts  to $\mathcal{M}_{a,b}(n+1)$.

Now consider the case $\ell=2$. If $b-2k \leqslant n-a-k+1$ there are $4$ lifts for $M$ to $\mathcal{M}_{a,b}(n+1)$ because again the third condition of \eqref{syss} is a consequence of the first two: notice that $(u \bmod 2^{n-a-k})$ determines $(u^2 \bmod 2^{n-a-k+1})$, and likewise for $y$. Suppose instead that $b-2k > n-a-k+1$. If $\delta$ is a square in $\Z_2^\times$ we can proceed as for $\ell$ odd, where we may suppose $v_2(U-\sqrt{\delta}Y)=b-2k-1$ and $v_2(U+\sqrt{\delta}Y)=1$ because $U-\sqrt{\delta}Y$ and $U+\sqrt{\delta}Y$ are even and  not both divisible by 4. Thus $M$ has 2 lifts to $\mathcal{M}_{a,b}(n+1)$. Finally, suppose that $\delta$ is not a square in $\Z_2^\times$, i.e. $\delta \not \equiv 1 \pmod 8$. 
For all $X,Y\in \Z_2$ lifting $x,y$ we know that $Y$ is odd, and we have $$v_2(X^2-dY^2)=2k+v_2(U^2-\delta Y^2)=2k+ \begin{cases} 1, \text{ if }\delta \equiv 3 \pmod 4 \\ 2, \text{ if } \delta \equiv 5 \pmod 8\,.  \end{cases}$$

Since $v_2(X^2-dY^2)$ is independent of $X,Y$, the matrix $M$ has $4$ lifts to $\mathcal{M}_{a,b}(n+1)$.
\end{proof}

\subsection{Normalizers of ramified Cartan subgroups}

Recall from Proposition \ref{complement} that $\mathcal{M}^*_{a,b}=\emptyset$ if $\ell$ is odd and $a>0$, or if $\ell=2$ and $a>1$.

\begin{thm}\label{thm:LiftsNormalizer}
Let $G$ be open in the normalizer of a ramified Cartan subgroup $C$ of $\GL_2(\Z_\ell)$. 
Let $n_0$ be the level of $G$. Assume $a=0$ if $\ell$ is odd, and $a\in \{0,1\}$ if $\ell=2$. Let $n\geqslant 1$.

If $\ell$ is odd, define $\mathcal{N}_{a,b}(n)$ as the subset of $G(n) \setminus C(n)$ consisting of those matrices $M$ that satisfy the following conditions: 
\begin{itemize}
\item[$\bullet$] $\det_\ell(M-I) \geqslant n$, if $n \leqslant b$; $\det_\ell(M-I)=b$, if $n>b$.
\end{itemize}

If $\ell=2$, define $\mathcal{N}_{a,b}(n)$ as the subset of $G(n) \setminus C(n)$ consisting of those matrices $M$ that satisfy the following conditions:
\begin{itemize}
\item[$\bullet$] $M \equiv I \pmod{2^a}$, $M \not \equiv I \pmod{2^{a+1}}$;
\item[$\bullet$] $\det_2(M-I) \geqslant n+1$, if $n<2a+b$; $\det_2(M-I) = 2a+b$, if $n \geqslant 2a+b$.
\end{itemize}

Define for $\ell$ odd and $\ell=2$ respectively:
\[
f(n)=\left\{\begin{array}{ll}
\ell & \text{if $n<b$} \\
\ell(\ell-1) & \text{if $n=b$} \\
\ell^2 & \text{if $n>b$} 
\end{array}\right.
\qquad
f(n) = \left\{\begin{array}{ll}
2 & \text{if $ n < 2a+b$} \\
4 & \text{if $n \geqslant 2a+b$\,.}
\end{array}\right.
\]

\begin{enumerate}
\item[(i)] For every $n \geqslant n_0$ we have $\#\mathcal{N}_{a,b}(n+1)=f(n)\cdot  \#\mathcal{N}_{a,b}(n)$. More precisely, for every matrix in $\mathcal{N}_{a,b}(n)$
the number of lifts to $\mathcal{N}_{a,b}(n+1)$ is $f(n)$.
\item[(ii)] If $n \geqslant n_0$ or if $n>a+b$ we have $\mathcal{M}^*_{a,b}(n)=\mathcal{N}_{a,b}(n)$.
\end{enumerate}

\end{thm}

\begin{proof}
We first prove (i). The parameters for $C$ are $(0,d)$, where $\ell \mid d$ if $\ell$ is odd, and by Lemma \ref{lemma:NormalizerAIsZero} any matrix in $G \setminus C$ is of the form 
\begin{equation}\label{nieuw}
M=\begin{pmatrix}
x & d y \\ -y & -x
\end{pmatrix}\,.
\end{equation}

\textit{The case $\ell$ odd ($n \geqslant n_0$ and $a=0$).} If $b<n$, every lift of a matrix in $\mathcal{N}_{0,b}(n)$ to $G(n+1)$ is in $\mathcal{N}_{0,b}(n+1)$.
If $b\geqslant n>0$ we proceed as for Theorem \ref{thm:GeneralLift}, noticing two facts: by Proposition \ref{complement} no matrix in $G\setminus C$ is congruent to the identity modulo $\ell$; the coefficient of $t_{11}$ in \eqref{eq:LiftsNormalizer} is nonzero because $\mathrm{det}_\ell(M-I)>0$ gives $x^2 \equiv 1 \pmod{\ell}$, and we have $n_{11}+1\equiv x \pmod{\ell}$.

\textit{The case $\ell=2$ ($n \geqslant n_0$ and $a\in \{0,1\}$).} Remark that $(M \bmod 2^n)$ determines $\det(M-I)$ modulo $2^{n+1}$. In particular, $\mathcal{N}_{a,b}(n)$ is well-defined.
Fix $M\in \mathcal{N}_{a,b}(n)$, and let $L$ be a lift of $M$ to $G(n+1)$.
Since $n\geqslant 2$, we know $L \equiv I \pmod{2^a}$ and $L \not \equiv I \pmod{2^{a+1}}$. If $n \geqslant 2a+b$, by the above remark all $4$ lifts of $M$ to $G(n+1)$ are in $\mathcal{N}_{a,b}(n+1)$. If $n < 2a+b$, we  have $\det_2(M-I)\geqslant n+1$ and hence $\det_2(L-I)\geqslant n+1$: writing any lift of $M$ in the form $L'=L+2^n T$ with $T$ as in \eqref{defT}, we are left to verify $\det_2(L'-I)\geqslant n+2$ for $n+1<2a+b$ and $\det_2(L'-I)= n+1$ for $n+1=2a+b$. We thus  study the inequality $\mathrm{det}_2\left( (L-I) + 2^n T\right) \geqslant n+2$
and an explicit verification (by Lemma \ref{lemma:NormalizerAIsZero} and because $2^{n+2}\mid 2^{2n}$) shows that there are precisely two lifts in $\mathcal{N}_{a,b}(n+1)$ as claimed.

We can prove (i)$\Rightarrow$(ii) as for Theorem \ref{thm:GeneralLift}: we clearly have  $f(n)\neq 0$ for all $n\geqslant 1$, and we have $\mathcal{N}_{a,b}(2a+b+1)=\mathcal{M}^*_{a,b}(2a+b+1)$ because the defining conditions hold for a matrix if and only if they hold for its lifts to $G$. 
\end{proof}

\section{Measures related to the $1$-Eigenspace}

\subsection{The case of $\GL_2(\Z_\ell)$ and unramified Cartan subgroups}

\begin{prop}\label{prop:FinitelyManyCab}
Suppose that $G$ is open either in $\GL_2(\Z_\ell)$ or in the normalizer of an unramified  Cartan subgroup of $\GL_2(\Z_\ell)$. Suppose $\mathcal{M}_{a,b}\neq \emptyset$. We have 
$\mathcal{M}_{a,b}(n_0)=\{I\}$ if $n_0 \leqslant a$ and 
$\mathcal{M}_{a,b}(n_0)=\mathcal{M}_{a,n_0-a}(n_0)$ if $a<n_0 \leqslant a+b$, and in particular we have:
$$\mu_{a,b}(n_0)=\left\{
\begin{array}{ll}
 \#G(n_0)^{-1} & \text{if $n_0 \leqslant a $}\\
\mu_{a,n_0-a}(n_0) & \text{if $a<n_0 \leqslant a+b$\,.}\\
\end{array}
  \right. $$
\end{prop}
\begin{proof}
For $n_0 \leqslant a$ the set $\mathcal{M}_{a,b}(n_0)$ contains at most the identity and it is non-empty by the assumption on $\mathcal{M}_{a,b}$. Now suppose $a<n_0 \leqslant a+b$. We claim that $\mathcal{M}_{a,n_0-a} \neq \emptyset$: the statement then follows from Theorem \ref{thm:GeneralLift} (ii) because by definition $\mathcal{H}_{a,b}(n_0)=\mathcal{H}_{a,n_0-a}(n_0)$.

We prove the claim by making use of Theorem \ref{thm:GeneralLift}. The assumption $\mathcal{M}_{a,b}\neq \emptyset$ implies that the set $\mathcal{M}_{a,b}(n_0)=\mathcal{H}_{a,b}(n_0)=\mathcal{H}_{a,n_0-a}(n_0)$ is nonempty. Since $n_0>a$ we have $f(n_0)\neq 0$ and hence $\mathcal{H}_{a,n_0-a}(n_0+1)$ is nonempty. This set equals $\mathcal{M}_{a,n_0-a}(n_0+1)$ because $n_0+1>a+(n_0-a)$.
\end{proof}

\begin{prop}\label{prop:GeneralCount}
Let $G'$ be either $\GL_2(\Z_\ell)$, an unramified Cartan subgroup of $\GL_2(\Z_\ell)$, or the normalizer of an unramified Cartan subgroup of $\GL_2(\Z_\ell)$. If $G$ is open in $G'$, we have:
\[
\mu_{a,b} = \mu_{a,b}(n_0) \cdot \left\{ \begin{array}{ll}
\#\mathbb{T}^{-(a+1-n_0)}\cdot  \#\mathbb{T}^\times & \text{if $n_0 \leqslant a, b=0$} \\
\#\mathbb{T}^{-(a+1-n_0)} \cdot (\#\mathbb{T}-\#\mathbb{T}^\times-1)\cdot  \ell^{-b}(\ell-1) & \text{if $n_0 \leqslant a, b>0$} \\
\ell^{-(a+b+1-n_0)} (\ell-1) & \text{if $a < n_0 \leqslant a+b$} \\
1 & \text{if $n_0>a+b$}\,.
\end{array}\right.
\]

We also have \[
\mu_{a,b} =  [G':G]\cdot  \#\mathbb{T}^{-a} \cdot \varepsilon \cdot \left\{ \begin{array}{ll}
1 & \text{if $n_0 \leqslant a, b=0$}  \\
\displaystyle \frac{\#\mathbb{T}-\#\mathbb{T}^\times-1}{\#\mathbb{T}^\times} \cdot \ell^{-b} (\ell-1) & \text{if $n_0 \leqslant a, b>0$}  \\
\end{array}\right.
\]
where $\varepsilon=\frac{1}{2}$ if $G'$ is the normalizer of a Cartan subgroup and $\varepsilon=1$ otherwise.
\end{prop}
\begin{proof}
To prove the first assertion we may suppose $\mathcal{M}_{a,b}\neq \emptyset$, because otherwise $\mu_{a,b}=\mu_{a,b}(n_0)=0$. The formula for $n_0>a+b$ has been proven in Lemma \ref{lemma:EverythingHasWellDefinedAB}. We have
\begin{equation*}
\#\mathcal{M}_{a,b}(a+b+1)=\#\mathcal{M}_{a,b}(n_0) \prod_{j=n_0}^{a+b} \frac{\#\mathcal{M}_{a,b}(j+1)}{\#\mathcal{M}_{a,b}(j)}
\end{equation*}
and by definition of $n_0$ we know $\#G(a+b+1)=\#G(n_0) \cdot \#\mathbb{T}^{a+b+1-n_0}$. We then obtain 
\begin{equation*}
\mu_{a,b} = \mu_{a,b}(n_0) \cdot \prod_{j=n_0}^{a+b} \#\mathbb{T}^{-1}\cdot  \frac{\#\mathcal{M}_{a,b}(j+1)}{\#\mathcal{M}_{a,b}(j)}
\end{equation*}
and the formulas for $n_0 \leqslant a+b$ can easily be deduced from Theorem \ref{thm:GeneralLift}.
We now turn to the second assertion. 
By Theorem \ref{thm:GeneralLift} (iii), $b=0$ implies $\mathcal M_{a,b}\neq  \emptyset$ while $b>0$ and $\mathcal M_{a,b}= \emptyset$ imply $\#\mathbb{T}-\#\mathbb{T}^\times-1=0$. In the latter case the formula for $\mu_{a,b}$ clearly holds, so we can assume $\mathcal M_{a,b}\neq \emptyset$ and hence 
$\mu_{a,b}(n_0)=\#G(n_0)^{-1}$ by Proposition \ref{prop:FinitelyManyCab}.
By Remark \ref{nor} and Lemma \ref{lemma:CardinalityCartans} (respectively, by Definition \ref{LieGL})
we know that $\#G'(1)=\varepsilon^{-1}\cdot \#\mathbb{T}^\times$ and $\#G'(n_0)=\#G'(1) \cdot \#\mathbb{T}^{n_0-1}$. We conclude because we have
\[
\#G(n_0)^{-1}=[G':G]\cdot  (\#G'(n_0))^{-1}=[G':G] \cdot \varepsilon \cdot (\#\mathbb{T}^\times)^{-1} \cdot \#\mathbb{T}^{1-n_0}.
\]
\end{proof}

\begin{exa}
Let $G$ be the inverse image in $\GL_2(\Z_2)$ of  
$$
G(2)=\langle
\begin{pmatrix}
3 & 3 \\ 0 & 1
\end{pmatrix}, \begin{pmatrix}
1 & 1 \\ 3 & 0
\end{pmatrix}
\rangle  \subset \GL_2(\Z/4\Z).
$$
Since $G$ has index $8$ and level $2$ in $\GL_2(\Z_2)$, by Proposition \ref{prop:FinitelyManyCab} we get $\mu_{a,b}(2)=1/12$ if $a\geqslant 2$ and $\mu_{a,b}(2)=\mu_{a,2-a}(2)$ if $a=0,1$ and $a+b\geqslant 2$. A direct computation gives $\mu_{0,0}(2)=1/3$, $\mu_{1,0}(2)=1/12$, $\mu_{0,2}(2)=1/2$ and 
$\mu_{0,1}(2)=\mu_{1,1}(2)=0$. So by Proposition \ref{prop:GeneralCount} we have:
\[
\mu_{a,b}=\left\{\begin{array}{lllll}
0 & \textit{ if $ a\in \{0,1\}, b=1$} \\ 
{1}/{3} & \textit{ if  $a=b=0$} \\ 
{1}/{12} & \textit{ if  $a=1, b=0$} \\ 
2^{-b} & \textit{ if $a=0, b \geqslant 2$} \\ 
8 \cdot 2^{-4a} & \textit{ if $a \geqslant 2, b=0$} \\  
12 \cdot 2^{-4a-b} & \textit{ if  $a \geqslant 2, b>0$}\,. 
\end{array}\right.
\]
\end{exa}

\begin{lem}\label{finitea}
Suppose that $G$ is open in a Cartan subgroup of $\GL_2(\Z_\ell)$. Let $n_0$ be the level of $G$. For all $a \geqslant n_0$ we have $\mu_{a,b}=\ell^{-2(a-n_0)}\mu_{n_0,b}$.
\end{lem}
\begin{proof}
We prove that for all $a\geqslant n_0$ we have $\mu(\mathcal{M}_{a+1,b})=\ell^{-2} \mu(\mathcal{M}_{a,b})$. We claim that the map 
\[
\begin{array}{cccc}
\phi: & \mathcal{M}_{a,b}(a+b+2) & \to & \mathcal{M}_{a+1,b}(a+b+2) \\
& M & \mapsto & I+\ell(M-I)
\end{array}
\]
is well-defined, surjective and $\ell^2$-to-$1$, so we have: 
\[
\mu(\mathcal{M}_{a+1,b})=\frac{\#\mathcal{M}_{a+1,b}(a+b+2) }{\#G(a+b+2)} = \frac{\ell^{-2} \#\mathcal{M}_{a,b}(a+b+2)}{\#G(a+b+2)}=\ell^{-2} \mu(\mathcal{M}_{a,b})\,.
\]
We are left to prove the claim. Since $a\geqslant n_0$, all matrices in the Cartan subgroup that are congruent to the identity modulo $\ell^a$ are in $G$ thus we may suppose that $G$ is the Cartan subgroup.
 A matrix $M\in G(a+b+2)$ is in the domain of $\phi$ if and only if the conditions in \eqref{conditions-Mab} hold, and these imply that $\phi(M)$ is in $\mathcal{M}_{a+1,b}(a+b+2)$ by \eqref{nfaa} and because we have:
$$\phi(M) \equiv I \pmod{\ell^{a+1}}\qquad  \phi(M )\not \equiv I \pmod{\ell^{a+2}}\qquad \mathrm{det}_{\ell}(\phi(M)-I)=2(a+1)+b.$$ If $N$ is in the codomain of $\phi$ then $I+\ell^{-1}(N-I)$ is well-defined modulo $\ell^{a+b+1}$: by Lemma \ref{condi-det} this matrix  belongs to $\mathcal{M}_{a,b}(a+b+1)$ and if $M$ is any lift of it to $\mathcal{M}_{a,b}(a+b+2)$ we have $\phi(M)=N$. 
This proves that $\phi$ is surjective (we may suppose that domain and codomain are nonempty, otherwise they must both be empty and the statement holds trivially). The set of preimages of $N$ consists of the matrices in $\mathcal{M}_{a,b}(a+b+2)$ congruent to $M$ modulo $\ell^{a+b+1}$, thus there are $\ell^2$ such preimages by Theorem \ref{thm:GeneralLift} (i)-(ii).
\end{proof}

\begin{rem}\label{rem:emptyness}
For every $a,b\geqslant 0$ we have $\mathcal M_{a,b}(\GL_2(\mathbb Z_\ell))\neq \emptyset$ because this set contains 
\[\qquad \begin{pmatrix} 2 & 1 \\ 1 & 1 \end{pmatrix} \quad\text{for $a=b=0$ ,\; and}\qquad
 \begin{pmatrix} 1 & \ell^{a+b}\\ \ell^a & 1 \end{pmatrix}\qquad\text{otherwise}.\]

If $C$ is a split Cartan subgroup of $\GL_2(\mathbb{Z}_\ell)$, we have $\mathcal M_{a,b}(C)\neq \emptyset$ for every $a,b\geqslant 0$, with the exception of $\ell=2$ and $a=0$:
 considering the diagonal model, if $\ell \neq 2$ or if $a \geqslant 1$ the set  $\mathcal M_{a,b}(C)$ contains 
$
\begin{pmatrix}
1+ \ell^{a}& 0\\
0 & 1+ \ell^{a+b}
\end{pmatrix};
$
however, for $\ell=2$ every diagonal invertible matrix is congruent to the identity modulo $2$.

If $C$ is a nonsplit Cartan subgroup of $\GL_2(\mathbb{Z}_\ell)$, we have $\mathcal M_{a,b}(C)= \emptyset$ for every $b>0$.
Indeed, if $M\in \mathcal M_{a,b}(C)$ then for $\ell$ odd (resp. $\ell=2$) we have 
$$\ell^{-a}(M-I)= \begin{pmatrix} z & dw \\ w & z \end{pmatrix}\qquad \text{resp.}\qquad 2^{-a}(M-I)= \begin{pmatrix} z & dw \\ w & z+w \end{pmatrix} $$ 
for some $z,w\in \Z_{\ell}$, and by Propositions \ref{prop:ConditionsCDodd} and \ref{prop:ConditionsCD2} these matrices are invertible unless $z$ and $w$ are  zero modulo $\ell$.
\end{rem}

\subsection{Ramified Cartan subgroups}

\begin{lem}\label{lem:ReductionCartan}
Suppose that $G$ is open in a Cartan subgroup $C$ of $\GL_2(\mathbb{Z}_\ell)$ with parameters $(0,d)$.
Write $d=m \ell^{v}$ with $\ell\nmid m$.
\begin{enumerate}
\item[\textit{(i)}] If $v$ is odd, we have $\mu_{a,b}=0$ for every $b>v$.
\item[\textit{(ii)}] If $v$ is even and $m$ is not a square in $\Z_\ell^\times$, we have $\mu_{a,b}=0$ for every $b>v+2$ (if $\ell$ is odd we have $\mu_{a,b}=0$ for $b>v$).
\item[\textit{(iii)}] If $v$ is even and $m$ is a square in $\Z_\ell^\times$, consider the  Cartan subgroup $C'$ of $\GL_2(\Z_\ell)$ with parameters $(0,1)$. There exists a closed subgroup $G_1$ of $C'$ such that the following holds: there is an explicit isomorphism between $G$ and $G_1$; the level of $G_1$ does not exceed the level of $G$ by more than ${v}/{2}$; 
for all $b>v$ the sets $\mathcal{M}_{a, b}(G)$ and 
$\mathcal{M}_{a+{v}/{2}, b-v}(G_1)$ have the same Haar measure in $G$ and $G_1$ respectively.
\end{enumerate}
\end{lem}

\begin{proof}
Fix $n>a+b+1$. By \eqref{nfaa} we can write any matrix in $ \mathcal M_{a,b}(n)$ as 
$$M=I+\ell^a \begin{pmatrix}
x & dy \\ y & x
\end{pmatrix}
$$
where $x,y\in \Z/\ell^{n-a}\Z$ satisfy $\lval(x^2-dy^2)=b$ and we have $\lval(x)=0$ or $\lval(y)=0$.

\emph{Proof of (i):} We have $\lval(x^2)\neq \lval(dy^2)$ and hence 
$v_\ell(x^2-dy^2)\leqslant v$, which implies that $\mathcal{M}_{a,b}(n)$ is empty for $b>v$.

\emph{Proof of (ii):} For $b>v$ we have $v_\ell(x^2)\geqslant v$ and we can write $b=v + v_\ell(x_1^2-my^2)$, where $x=\ell^{v/2} x_1$. We must have $x_1^2 \equiv my^2 \pmod {\ell}$, which is impossible for $\ell$ odd because $m$ is not a square modulo $\ell$.
If $\ell=2$ and $b>v+2$ we should similarly have $x_1^2 \equiv my^2 \pmod {8}$, which is impossible because $m$ is not a square modulo $8$.

\emph{Proof of (iii):} Since $d$ is a square in $\Z_\ell$, we may fix a square root $\sqrt{d}$ of it. Define $G_1$ to be the image of 
\[
\begin{array}{cccc}
\phi: & G & \to & C' 
\\ & I+\ell^a \begin{pmatrix}
x & d y \\ y & x
\end{pmatrix} & \mapsto & I+\ell^a \begin{pmatrix}
x &  y \sqrt{d} \\   y \sqrt{d} & x
\end{pmatrix}\,.
\end{array}
\]
Embedding $G$ in $\GL_2(\mathbb Q_{\ell})$, the map $\phi$ can be identified with the conjugation by $\begin{pmatrix}
1 & 0 \\ 0 & \sqrt{d}
\end{pmatrix}
$, thus $\phi$ is a continuous group isomorphism between $G$ and $G'$, and for every $M\in G$ we have $\det \phi(M)=\det(M)$ and $\det (\phi(M)-I)=\det(M-I)$.
Let $n_0$ denote the level of $G$. The level of $G_1$ is at most $n_0+v/2$ because any  matrix in $C'$ which is congruent to the identity modulo $\ell^{n_0+v/2}$ is the image via $\phi$ of a matrix in $C$ that is congruent to the identity modulo $\ell^{n_0}$:
$$I+\ell^{n_0} \begin{pmatrix} \ell^{v/2} z &  y d \\ y & \ell^{v/2} z \end{pmatrix} \stackrel{\phi}{\longrightarrow} I+\ell^{n_0+v/2} \begin{pmatrix} z & y\sqrt{m} \\ y\sqrt{m} & z \end{pmatrix}\,. $$

If $b>v$, we have $v_\ell(x^2)\geqslant v$ and a straightforward verification shows that $\phi$ induces a bijection from $\mathcal{M}_{a,b}(G)$ to $\mathcal{M}_{a+v/2,b-v}(G_1)$, and we conclude because a continuous isomorphism of profinite groups preserves the normalized Haar measure, see \cite[Proposition 18.2.2]{MR2445111}.\end{proof}

\begin{lem}\label{lem:FakeSplitCartan}
If $G$ is open in the Cartan subgroup of $\GL_2(\Z_2)$ with parameters $(0,1)$ the sets $\mathcal{M}_{a,1}$ and $\mathcal{M}_{a,2}$ are empty. Moreover, there exists an open subgroup $G_1$ of the subgroup of diagonal matrices in $\GL_2(\Z_2)$ such that the following holds: there is an explicit isomorphism between $G$ and $G_1$; the level of $G_1$ does not exceed the level of $G$ by more than 1; for all $b>2$ the sets $\mathcal{M}_{a,b}(G)$ and $\mathcal{M}_{a+1,b-2}(G_1)$ have the same Haar measure in $G$ and $G_1$ respectively. 
\end{lem}

\begin{proof}
We can write any matrix in $\mathcal{M}_{a,b}(G)$ as 
\begin{equation}\label{proofM}
M=I + 2^a \begin{pmatrix}
x & y \\ y & x
\end{pmatrix}
\end{equation}
where at least one between $x$ and $y$ is a $2$-adic unit. Working modulo $8$, we see that $b=v_2(x^2-y^2)$ cannot be $1$ or $2$.

We sketch the rest of the proof, which mimics Lemma \ref{lem:ReductionCartan} (iii). We define a map $\phi$ from $G$ to $\GL_2(\Z_2)$, denoting $G_1$ its image:  
\begin{equation}\label{proofM'}
\phi(M)=I + 2^a \begin{pmatrix}
x+y & 0 \\ 0 & x-y
\end{pmatrix}.
\end{equation}
We clearly have $\det_2(\phi(M)-I)=\det_2(M-I)$. If $b>2$, then $x+y$ and $x-y$ must be even and not both divisible by $4$, and it follows that $\phi(M)\in \mathcal{M}_{a+1,b-2}(G_1)$.
\end{proof}

\begin{lem}\label{lem:MabN0IsIndependentOfB}
Let $G$ be open in a Cartan subgroup $C$ of $\GL_2(\mathbb{Z}_\ell)$ with parameters $(0,d)$, where $d$ is a square in $\Z_\ell$. 
For any fixed value of $a$, the set $\mathcal{M}_{a,b}(n_0)$ does not depend on $b$ provided that $b\geqslant b_0$, where $b_0:=\max\{1+v_\ell(4d), n_0-a+v_\ell(2d)\}$.
\end{lem}
\begin{proof}
Let $b \geqslant b_0>0$ and consider a matrix in $\mathcal{M}_{a,b}$:
$M=I+\ell^a \begin{pmatrix}
x & dy \\ y & x
\end{pmatrix}$.
It suffices to show that for every $b' \geqslant b_0$ there is $M' \in \mathcal{M}_{a,b'}$ that is congruent to $M$ modulo $\ell^{n_0}$.
We have $v_\ell(x^2-d y^2)\geqslant b>v_\ell(d)$ and at least one among $x$ and $y$ is a unit. One checks easily that $y$ cannot be divisible by $\ell$, so we have $v_\ell(x^2)= v_\ell(d)$ and we can define $y''=x/\sqrt{d}$. Given two units in $\mathbb Z_\ell$, either their sum or their difference has valuation $v_\ell(2)$, so up to replacing $\sqrt{d}$ by $-\sqrt{d}$ we get $v_\ell(x-\sqrt{d} y)\geqslant b-v_\ell(d)/2-v_\ell(2)\geqslant n_0-a+v_\ell(d)/2$ and hence $y'' \equiv y \pmod{\ell^{n_0-a}}$. Defining $B:=b'-v_\ell(d)/2-v_\ell(2)\geqslant n_0-a$, the matrix
$$M'=I+\ell^a \begin{pmatrix}
x+\ell^B & dy'' \\ y'' & x +\ell^B
\end{pmatrix}$$ is congruent to $M$ modulo $\ell^{n_0}$ and we have $\det(M'-I)=\ell^{2a}(2x\ell^B+\ell^{2B})$.
Since $B>v_\ell(d)/2+v_\ell(2)=v_\ell(x)+v_\ell(2)$ we have $\det_\ell(M'-I)=2a+b'$ and hence $M'\in \mathcal{M}_{a,b'}$.
\end{proof}

\subsection{Normalizers of Cartan subgroups}

Recall the notation from Section \ref{star}.

\begin{thm}\label{thm:CountCPrime}
Let $G$ be open in the normalizer of a Cartan subgroup of $\GL_2(\mathbb{Z}_\ell)$. Let $n_0$ be the level of $G$. 
\begin{enumerate}
\item[(i)] If $\ell$ is odd or $C$ is unramified, we have: 
$$\mu(\mathcal M^*_{a,b})= \left\{ 
\begin{array}{lllll}
0 & \text{if $a>0$} \\
\mu(\mathcal M^*_{0,b}(n_0)) & \text{if $a=0, b < n_0$} \\
\mu(\mathcal M^*_{0,n_0}(n_0)) \cdot (\ell-1) \cdot \ell^{n_0 -b- 1} & \text{if $a=0, b  \geqslant n_0$\,.}
\end{array}\right.$$

\item[(ii)] If $\ell=2$ and $C$ is ramified, we have: 
$$\mu(\mathcal M^*_{a,b})= \left\{ 
\begin{array}{lllll}
0 & \text{if $a> 1$} \\
\mu(\mathcal M^*_{a,b}(n_0)) & \text{if $a\leqslant 1$ and $2a+b \leqslant n_0$} \\
\mu(\mathcal M^*_{a,n_0+1}(n_0)) \cdot 2^{n_0-2a-b} & \text{if $a\leqslant 1$ and $2a+b >  n_0$\,.}
\end{array}\right.$$ 
\end{enumerate}
\end{thm}

\begin{proof}
For the first assertion, by Proposition  \ref{complement} we have $\mu(\mathcal M^*_{a,b})=0$ for $a>0$.
For $b<n_0$ we clearly have $\mu(\mathcal M^*_{0,b})=\mu(\mathcal M^*_{0,b}(n_0))$. 
If $b \geqslant n_0$, Theorem \ref{thm:LiftsNormalizer} (ii) implies 
\[
\mathcal{M}^*_{0,b}(n_0) = \{ M \in (G \setminus C)(n_0) \bigm\vert \mathrm{det}_\ell(M-I) \geqslant n_0 \}=\mathcal{M}_{0,n_0}^*(n_0)
\]
so in particular we have $\mu(\mathcal M^*_{0,b}(n_0))=\mu(\mathcal M^*_{0,n_0}(n_0))$. We conclude by Theorem \ref{thm:GeneralLift} and \ref{thm:LiftsNormalizer} respectively, by the same argument used to prove Proposition \ref{prop:GeneralCount}.
The second assertion follows analogously from Proposition  \ref{complement} and Theorem \ref{thm:LiftsNormalizer}. Indeed, if $2a+b > n_0$ then $\mathcal{M}^*_{a,b}(n_0)$ is independent of $b$ by Theorem \ref{thm:LiftsNormalizer} (ii).
\end{proof}

\begin{cor}\label{cor:AsymptoticGrowthCartan}
Let $G$ be open in the normalizer $N$ of a Cartan subgroup $C$ of $\GL_2(\Z_\ell)$.
For $a \in \{0,1\}$ there exist (effectively computable) rational numbers $c_1(a), c_2(a), c_3(a)$ such that 
\[
\mu(\mathcal{M}_{a,b}) = c_1(a) \ell^{-b}, \qquad \mu( \mathcal{M}_{a,b} \cap (N\setminus C)) = c_2(a) \ell^{-b}, \qquad \mu( \mathcal{M}_{a,b} \cap C) = c_3(a) \ell^{-b}
\]
hold for all sufficiently large $b$ (and the bound is effective). The rational constants $c_i(a)$ may depend on $\ell$ and $G$, as well as on $a$.
\end{cor}

\begin{proof}
The assertion for $\mathcal{M}_{a,b}$ follows from the other two, and the assertion for $\mathcal{M}_{a,b} \cap (N\setminus C)$ holds by Theorem \ref{thm:CountCPrime}. 
Now consider $\mathcal{M}_{a,b} \cap C$. Because of Lemmas \ref{lem:ReductionCartan} and \ref{lem:FakeSplitCartan}, we only need to consider the case when $C$ is a split Cartan subgroup. We apply Proposition \ref{prop:FinitelyManyCab} (in view of Remark \ref{rem:emptyness}) to deduce that $\mu_{a,b}(n_0)$ is constant for $b\geqslant n_0$ and then apply Proposition \ref{prop:GeneralCount}.
\end{proof}

\section{The results of the Introduction}

\subsection{Proof of Theorem \ref{main-thm}}

\begin{defi}\label{admissible}
We call a subset of $\mathbb N^2$ \emph{admissible} if it is the product of two subsets of $\mathbb N$, each of which is either finite or consists of all integers greater than some given one. The family of finite unions of admissible sets is closed w.r.t. intersection, union and complement.
\end{defi}

We describe a general computational strategy to determine $\mu_{a,b}$ for all $a,b\geqslant 0$. 
Depending on the input data (i.e.\@ a finite amount of information about the group $G$), we can choose which of the previous results must be applied, and we can compute the finitely many rational parameters that appear in the statements. After a case distinction, we have formulas for all measures $\mu_{a,b}$ that depend only on $a$, $b$, and finitely many known constants. As it can be seen from the explicit description below, the cases give a partition of $\mathbb N^2$ into finitely many admissible subsets and on each of them the formula for $\mu_{a,b}$  is as requested.
We first need to express the relevant properties of $G$ in terms of finitely many parameters:
\begin{enumerate}
\item The group $G$ is open in $G'$, which is either $\GL_2(\Z_\ell)$, 
a Cartan subgroup, or the normalizer of a Cartan subgroup. We describe a Cartan subgroup  with the integer parameters $(c,d)$ of Section \ref{subsec-classification}, which also determine whether this is split, nonsplit or ramified. The cardinality of the tangent space $\mathbb T$ and of its subset $\mathbb T^{\times}$ is known, see Section \ref{subsec-Lie}.

\item We fix an integer $n_0  \geqslant 1$ such that $G$ is the inverse image in $G'$ of $G(n_0)$ for the reduction modulo $\ell^{n_0}$. If $\ell=2$ and $G'$ is (the normalizer of) a ramified Cartan subgroup, we take $n_0 \geqslant 2$ ($n_0$ is not necessarily the level of $G$, see Remark \ref{rem:Level}).

\item We need to know the finite group $G(n_0)$ explicitly. From this we extract various data, including the order of $G(n_0)$, the index $[G'(n_0): G(n_0)]=[G':G]$, and the following information: for each of the finitely many pairs $(a,b)$ such that $a<n_0$ and $b \leqslant n_0-a$, we need to know the counting measure $\mu_{a,b}(n_0)$ and whether the set $G(n_0)\cap \mathcal M_{a,b}(G'; n_0)$ is empty or not. For (normalizers of) ramified Cartan subgroups we may also need finitely many other quantities which can all be read off $G(n_0)$, see the description below.
\end{enumerate}

We make repeated use of the following remark: suppose that for $(a,b)$ in some admissible set $S=A \times B$ with $A$ finite we have $\mu_{a,b}=c(a) \ell^{-b}$, where $c(a)$ is a rational number depending on $a$. Then $S$ is the finite union of the sets $S_a=\{a\} \times B$, and by choosing the constant $c'(a)$ appropriately we have $\mu_{a,b} = c'(a) \ell^{-a \dim(G)-b}$ for all $(a,b) \in S_a$.

\emph{If $G'=\GL_2(\Z_\ell)$:}
We can compute the values $\mu_{a,b}$ for all pairs $(a,b)$ such that $\mathcal M_{a,b}\neq \emptyset$ by Propositions \ref{prop:FinitelyManyCab} and \ref{prop:GeneralCount}. 
Up to refining the partition, we can ensure that $\mu_{a,b}$ is a constant multiple of $\ell^{-4a-b}$ on every set of the partition.

We are left to determine the pairs $(a,b)$ such that $\mathcal M_{a,b}= \emptyset$ (and hence $\mu_{a,b}=0$) and show that they form an admissible subset of $\mathbb N^2$.
By Remark \ref{rem:emptyness} we know $\mathcal{M}_{a,b}(G') \neq \emptyset$, and by Remark \ref{vuoto}  we just need to know whether $G(n_0)\cap \mathcal M_{a,b}(G'; n_0)$ is empty. By Proposition \ref{prop:FinitelyManyCab} (applied to $G'$) there are only finitely many distinct sets of the form $\mathcal M_{a,b}(G'; n_0)$ to consider and it is a finite computation to determine those that intersect $G(n_0)$ trivially.

\emph{If $G'$ is a nonsplit Cartan subgroup:}
By Lemma \ref{finitea} and Remark \ref{rem:emptyness} we reduce to the case $a\leqslant n_0$ and $b=0$. Thus by Proposition \ref{prop:GeneralCount} we only need to evaluate $\mu_{a,0}(n_0)$ for $a \leqslant n_0$. Since we only have finitely many values of $a$ to consider, up to refining the partition we find that $\mu_{a,b}$ is a constant multiple of  $\ell^{-2a-b}$ on every set of the partition.

\emph{If $G'$ is a split Cartan subgroup:}
By Lemma \ref{finitea} we reduce to the case $a\leqslant n_0$, so fix one of those finitely many values for $a$. By Proposition \ref{prop:GeneralCount} it suffices to evaluate $\mu_{a,b}(n_0)$ for all $b\geqslant 0$. If $\mathcal M_{a,b}\neq \emptyset$, by Proposition \ref{prop:FinitelyManyCab} we only need to consider the finitely many cases for which $b\leqslant n_0$.

We are left to determine the pairs $(a,b)$ such that $\mathcal M_{a,b}= \emptyset$ (and hence $\mu_{a,b}=0$) and show that they form an admissible subset of $\mathbb N^2$. By Remark \ref{rem:emptyness} the set $\mathcal{M}_{a,b}(G')$ is empty (and hence $\mathcal M_{a,b}= \emptyset$) if and only if $\ell=2$ and $a=0$. In the remaining cases we have $\mathcal{M}_{a,b}(G') \neq \emptyset$ and $a\leqslant n_0$, so we may reason as for $G'=\GL_2(\Z_\ell)$.
Up to refining the partition, we find that $\mu_{a,b}$ is a constant multiple of $\ell^{-2a-b}$ on every set of the partition.

\emph{If $G'$ is a ramified Cartan subgroup:}
By Lemma \ref{finitea} we reduce to the case $a\leqslant n_0$, so fix one of these finitely many values for $a$. By Propositions \ref{prop:ConditionsCDodd} and \ref{prop:ConditionsCD2}, the parameters for $C$ are $(0,d)$ and we can apply Lemma \ref{lem:ReductionCartan}.
If we are in cases (i)-(ii) of this lemma, we only need to consider the finitely many values $b\leqslant v_\ell (d)+2$. The measure $\mu_{a,b}$ for a single pair $(a,b)$ can be computed explicitly as $\mu_{a,b}(a+b+1)$. Notice that the group $G(a+b+1)$ and hence its subset $\mathcal{M}_{a,b}(a+b+1)$ can be determined from the knowledge of $G'$ and $G(n_0)$.
Now suppose that we are in case (iii) of Lemma \ref{lem:ReductionCartan}. Recalling that $a\leqslant n_0$ is fixed, we may compute the finitely many measures $\mu_{a,b}$ where $b\leqslant v_\ell (d)$. For $b> v_\ell (d)$ we reduce to a similar problem for an unramified Cartan subgroup: if $\ell$ is odd, the Cartan subgroup with parameters $(0,1)$ is unramified; if $\ell=2$ we further apply Lemma \ref{lem:FakeSplitCartan}.
Once more, this gives a partition as requested.

\emph{The case when $G'$ is the normalizer of a Cartan subgroup:}
As shown in Section \ref{star}, to reduce to the case when $G'$ is a Cartan subgroup it suffices to compute the measures $\mu(\mathcal M^*_{a,b})$ for all $a,b\geqslant 0$. We achieve this by Theorem \ref{thm:CountCPrime}: it suffices to compute $\mu(\mathcal M^*_{a,b}(n_0))$ for finitely many pairs $(a,b)$. 
The measures in the Cartan subgroup and those related to its complement in the normalizer add up to an expression of the desired form, because they can both be written as $\ell^{-2a-b}$ times a constant.

\subsection{The special case where $G$ has index $1$}

\begin{proof}[Proof of Theorem \ref{thm-countGL2}]
Consider Definition \ref{LieGL} and Proposition \ref{prop:GeneralCount}. The cases with $a>0$ are clear because $n_0=1 \leqslant a$. If $a=b=0$, we have  $\mu_{0,0}=\mu_{0,0}(1)=\#\mathcal{M}_{0,0}(1)/\#\GL_2(\Z/\ell\Z)$ so it suffices to prove 
$\#\mathcal{M}_{0,0}(1)=\ell(\ell^3-2\ell^2-\ell+3)\,.$
Equivalently, we have to show that there are $\#\GL_2(\Z/\ell\Z)-\#\mathcal{M}_{0,0}(1)=\ell^3-2\ell$ matrices in  $\GL_2(\Z/\ell\Z)$ that have $1$ as an eigenvalue. This is done e.g. in the course of \cite[Proof of Theorem 5.5]{MR2640290} (see also \cite[Section 4]{MR1995144}), but for the convenience of the reader we sketch the computation. Matrices admitting 1 as an eigenvalue are the identity and those that are conjugate to one of the following: 
 $$J_1=\begin{pmatrix} 1& 1\\ 0& 1\end{pmatrix}\qquad J_{\lambda}=\begin{pmatrix} 1& 0\\ 0& \lambda \end{pmatrix}\qquad \lambda \neq 0,1\,.$$
Since the centralizer of $J_1$ has size $\ell(\ell-1)$ while that of $J_\lambda$ has size $(\ell-1)^2$,  
we may conclude by computing the size of the conjugacy classes as the index of the centralizer. 

If $a=0$ and $b>0$, we have to evaluate $(\ell-1)\cdot \ell^{-b}\cdot \mu_{0,b}(1)$ for $b>0$. By Proposition \ref{prop:FinitelyManyCab} and Remark \ref{rem:emptyness} we have 
$\mu_{0,b}(1)=\mu_{0,1}(1)=\#\mathcal{M}_{0,1}(1)/\#\GL_2(\Z/\ell\Z)$ so it suffices to prove 
$$\#\mathcal{M}_{0,1}(1)=(\ell^2-\ell-1)(\ell+1)\,.$$
We may conclude by noticing that $\mathcal{M}_{0,1}(1)$ consists of the $\ell^3-2\ell$ matrices in $\#\GL_2(\Z/\ell\Z)$ that have $1$ as an eigenvalue, with the exception of the identity matrix.
\end{proof}

\begin{proof}[Proof of Theorem \ref{thm-countsplitnonsplitCartan}]
We can take $n_0=1$ so the cases with $a\geqslant 1$ follow immediately from Proposition \ref{prop:GeneralCount} and \eqref{eqprop:XB}. Now suppose $a=0$.
Consider a split Cartan subgroup with the diagonal model. For $b=0$, in order to evaluate $\mu_{0,0}(1)$ we count those diagonal matrices $\operatorname{diag}(x,y)$ such that $xy$ and $(x-1)(y-1)$ are in $(\Z/\ell\Z)^\times$: this means $x, y \not \equiv 0, 1 \pmod \ell$, hence there are $(\ell-2)^2$ choices, out of $(\ell-1)^2$ total elements. 
For $b>0$ we have $\mu_{0,b}(1)=\mu_{0,1}(1)$ by Proposition \ref{prop:FinitelyManyCab} and Remark \ref{rem:emptyness}. There are $2(\ell-2)$ diagonal matrices in $\GL_2(\Z/\ell\Z)$ such that exactly one of the two diagonal entries is congruent to 1 modulo $\ell$, so we get by Proposition \ref{prop:GeneralCount}:
\[
\mu_{0,b}=\mu_{0,b}(1) \ell^{-b} (\ell-1)  = \frac{2(\ell-2)}{(\ell-1)^2} (\ell-1) \ell^{-b}\,.
\]

Now consider the nonsplit case. By Remark \ref{rem:emptyness} we know $\mathcal{M}_{a,b}=\emptyset$ for $b>0$. For $b=0$ we need to evaluate $\mu_{0,0}$: by Lemma \ref{lemma:EverythingHasWellDefinedAB} and by the previous case $a \geqslant 1$ we have  
\[
1 = \sum_{a,b \geqslant 0} \mu_{a,b} = \sum_{a \geqslant 0} \mu_{a,0}= \mu_{0,0} + \sum_{a \geqslant 1} \ell^{-2a}\,.
\]
\end{proof}
\begin{proof}[Proof of Theorem \ref{thm-countnormalizerCartan}] 

Let $C$ be the Cartan subgroup and let $C'$ be as in Lemma \ref{lem-Norm}. Fixing some $n> a+b$ we get 
$$\mu_{a,b}=\mu_{a,b}(n)=\frac{\#\mathcal{M}_{a,b}(n)}{\#(C\cup C')(n)}=\frac{\#(\mathcal{M}_{a,b}(n)\cap C(n))+\#(\mathcal{M}_{a,b}(n)\cap C'(n))}{2\cdot \#C(n)}\,.$$
By definition we have $\mu^C_{a,b}=\#(\mathcal{M}_{a,b}(n)\cap C(n))/\#C(n)$, so it suffices to show $\mu^*_{a,b}=\#(\mathcal{M}_{a,b}(n)\cap C'(n))/\#C(n)$.
If $a>0$ then no matrix in $C'(n)$ is in $\mathcal{M}_{a,b}(n)$ by Lemma \ref{lemma:NormalizerAIsZero}. For $a=0$ we are left to prove that for $n= b+1$ we have $\#\mathcal{M}_{0,b}(n)\cap C'(n)=\mu^*_{0,b} \cdot \#C(n)$. By Lemma \ref{lemma:NormalizerAIsZero} the elements of $C'(b+1)$ are those matrices of the form 
\begin{equation}\label{form}
M=\begin{pmatrix}
\alpha & -d \beta + c \alpha \\ \beta & -\alpha
\end{pmatrix} \qquad \alpha, \beta \in \Z/\ell^{b+1}\Z
\end{equation}
where $c, d$ are here the reductions modulo $\ell^{b+1}$ of the parameters of $C$.
Thus we need to count the pairs $(\alpha, \beta) \in (\Z/\ell^{b+1}\Z)^2$ satisfying 
\begin{equation}\label{proof1}
\mathrm{det}_\ell (M-I)=\lval(1-\alpha^2+d\beta^2 -c \alpha \beta)=b\,. 
\end{equation}
We also need $\mathrm{det}_\ell (M)=\lval(-\alpha^2+d\beta^2 -c \alpha \beta)=0$, which for $b>0$ follows from \eqref{proof1}.

The count for the split case will give $(\ell-1)(\ell-2)$ for $b=0$ and $(\ell-1)^2\ell^b$ for $b>0$. The count for the nonsplit case will give $(\ell+1)(\ell-2)$ for $b=0$ and $(\ell^2-1)\ell^{b}$ for $b>0$. We then conclude by Lemma \ref{lemma:CardinalityCartans} because $\# C(b+1)$ equals $(\ell-1)^2\ell^{2b}$ and $(\ell^2-1)\ell^{2b}$ for the split and the nonsplit case respectively.
One can easily check that the affine curve $\mathcal{D}: 1-x^2+y(dy-cx)=0$ (defined over $\Z_\ell$) is smooth over $\Z/\ell\Z$. We have 
$\#\mathcal{D}(\Z/\ell\Z)=\ell \pm 1$, where the sign is $-$ (resp.~$+$) if $C$ is split (resp.~nonsplit). Indeed, $\mathcal{D}$ can be identified over $\Z/\ell\Z$ with the open subscheme of $\{Z^2-X^2+Y(dY-cX)=0\} \cong \mathbb{P}^1$ given by $Z \neq 0$, and by Propositions \ref{prop:ConditionsCDodd} and \ref{prop:ConditionsCD2} there are two   (resp. zero) $\Z/\ell\Z$-points with $Z=0$ if $C$ is split (resp. nonsplit). 

\emph{The case $b=0$.} There are precisely $\ell^2-(\ell \pm 1)$ pairs $(\alpha, \beta)\in (\Z/\ell \Z)^2$ that do {not} correspond to points in $\mathcal{D}(\Z/\ell\Z)$. Since we only want invertible matrices, we need to exclude those pairs such that $-\alpha^2+d\beta^2-c\alpha\beta =0$. 
By Propositions \ref{prop:ConditionsCDodd} and \ref{prop:ConditionsCD2} this equation has $2\ell-1$ solutions if $C$ is split and has only the trivial solution $\alpha= \beta=0$ if $C$ is nonsplit.

\emph{The case $b>0$.} As $\mathcal{D}$ is smooth over $\mathbb{F}_\ell$, by (the higher-dimensional version of) Hensel's Lemma \cite[Proposition 7.8]{Nekovar} we have $\#\mathcal{D}(\Z/\ell^b\Z)=\ell^{b-1}\cdot \#\mathcal{D}(\Z/\ell\Z)$. A pair $(\alpha, \beta)\in (\Z/\ell^{b+1}\Z)^2$ as in \eqref{proof1} reduces to a point in $\mathcal{D}(\Z/\ell^b\Z)$, so it suffices to prove that there are precisely $\ell^2-\ell$ pairs $(\alpha, \beta)$ as in \eqref{proof1} that lie over some fixed $(\overline{\alpha}, \overline{\beta}) \in \mathcal{D}(\Z/\ell^b\Z)$. There are $\ell^2$ lifts of $(\overline{\alpha}, \overline{\beta})$ to $(\Z/\ell^{b+1}\Z)^2$ and we must avoid those in $\mathcal{D}(\Z/\ell^{b+1}\Z)$, which are exactly $\ell$ again by Hensel's Lemma.
\end{proof}

\bibliographystyle{abbrv}
\bibliography{biblio}

\subsection*{Acknowledgements} 
We thank the referee for their useful suggestions. The second author gratefully acknowledges financial support from the SFB-Higher Invariants at the University of Regensburg.

\end{document}